\documentclass[11pt, a4paper,final]{article}
\usepackage{a4wide}
\usepackage[T1]{fontenc}
\usepackage{ae,aecompl}
\usepackage[activate={true,nocompatibility},final,tracking=true,kerning=true,spacing=true,stretch=10,shrink=10]{microtype}
\microtypecontext{spacing=nonfrench}
\usepackage[utf8]{inputenc}

\usepackage[british]{babel}
\usepackage{csquotes}

\usepackage{amsthm}
\usepackage{amsmath, amssymb}
\usepackage{tikz}
\usepackage{tikz-cd}
\usetikzlibrary{patterns}
\usepackage{ifthen}

\usepackage[draft=false]{hyperref}
\usepackage{xcolor}
\hypersetup{
	colorlinks,
	linkcolor={red!50!black},
	citecolor={blue!50!black},
	urlcolor={blue!80!black}
}
\usepackage[backend=biber,
url=false,
isbn=false,
backref=true,
citestyle=alphabetic,
bibstyle=alphabetic,
autocite=inline,
maxnames=99,
minalphanames=4,
maxalphanames=4,
sorting=nyt,]{biblatex}
\usepackage[inline]{enumitem}

\theoremstyle{plain}
\newtheorem*{mainthm}{Main Theorem}
\newtheorem{theorem}{Theorem}[section]
\newtheorem{corollary}[theorem]{Corollary}
\newtheorem{lemma}[theorem]{Lemma}
\newtheorem{fact}[theorem]{Fact}
\newtheorem{proposition}[theorem]{Proposition}

\newtheorem{claim}{Claim}
\newtheorem*{claim*}{Claim}
\usepackage{etoolbox}
\AtEndEnvironment{proof}{\setcounter{claim}{0}}
\newenvironment{clmproof}[1][\proofname]{\proof[#1]}{\endproof}
\theoremstyle{definition}
\newtheorem{definition}[theorem]{Definition}
\theoremstyle{remark}
\newtheorem{remark}[theorem]{Remark}

\newtheorem{question}[theorem]{Question}

\DeclareMathOperator{\Aut}{Aut}

\DeclareMathOperator{\Age}{Age}

\DeclareMathOperator{\dom}{dom}
\DeclareMathOperator{\range}{range}

\DeclareMathOperator{\tp}{tp}
\DeclareMathOperator{\qftp}{qftp}
\DeclareMathOperator{\otp}{otp}

\newcommand{\meet}{\mathbin{\wedge}}
\newcommand{\cK}{\mathcal K}

\newcommand{\bT}{\mathbf T}
\newcommand{\isdefined}{{\downarrow}}
\newcommand{\isundefined}{{\uparrow}}

\newcommand{\bQ}{{\mathbf Q}}
\newcommand{\totalaut}{{\mathrm{t}}}

\renewcommand{\restriction}{\mathord{\upharpoonright}}

\newcommand{\xqed}[1]{%
	\leavevmode\unskip\penalty9999 \hbox{}\nobreak\hfill
	\quad\hbox{\ensuremath{#1}}}

\makeatletter
\renewcommand*\author[1]{%
	\stepcounter{author}%
	\ifnum\c@author=1
	\gdef\@author{#1}%
	\else
	\xdef\@author{\unexpanded\expandafter{\@author\and#1}}%
	\fi
	\csgdef{author@\the\c@author}{#1}}
\newcommand*\grant[1]{%
	\xdef\@author{\unexpanded\expandafter{\@author\footnote{#1}}}
}
\newcommand*\email[1]{%
	\csgdef{email@\the\c@author}{#1}}
\newcommand*\orcid[1]{%
	\csgdef{orcid@\the\c@author}{#1}}
\newcommand*\address[1]{%
	\csgdef{address@\the\c@author}{#1}}
\AtEndDocument{%
	\xdef\author@count{\the\c@author}%
	\c@author=1
	\print@authors}
\newcommand*\print@authors{%
	\ifnum\c@author>\author@count
	\else
	\print@author{\the\c@author}%
	\advance\c@author by 1
	\expandafter\print@authors
	\fi}
\newcommand*\print@author[1]{%
	\par\medskip
	\begin{tabular}{@{}l@{}}%
		\textsc{\csuse{author@#1}}\\
		\csuse{address@#1}\\
		\href{\csuse{orcid@#1}}{\includegraphics[height=\fontcharht\font`\B]{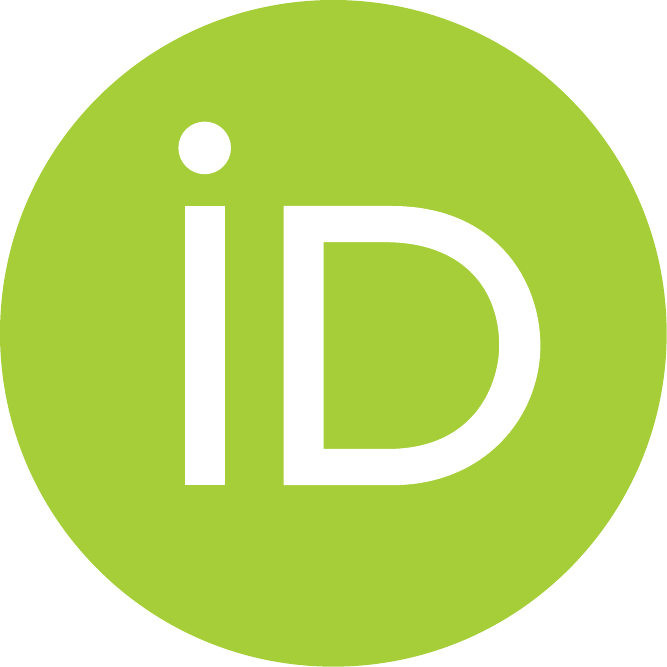} \csuse{orcid@#1}}\\
		\href{mailto:\csuse{email@#1}}{\csuse{email@#1}}
\end{tabular}}
\makeatother
\makeatletter
\newcommand{\subjclass}[2][1991]{%
	\let\@oldtitle\@title%
	\gdef\@title{\@oldtitle\footnotetext{#1 \emph{Mathematics subject classification.} #2}}%
}
\newcommand{\keywords}[1]{%
	\let\@@oldtitle\@title%
	\gdef\@title{\@@oldtitle\footnotetext{\emph{Key words and phrases.} #1.}}%
}
\makeatother

\title{On the automorphism group of the universal homogeneous meet-tree}
\subjclass[2020]{03C15, 03E15, 06A12, 20E08, 54H11, 20B27}
\keywords{Fraïssé limit, generic automorphism, tree, Polish group, ample generics}

\author{Itay Kaplan}
\grant{The first author would like to thank the Israel Science Foundation for partial support of this research (grants no.\ 1533/14 and 1254/18).}
\address{The Hebrew University of Jerusalem}
\orcid{https://orcid.org/0000-0002-7032-1710}
\email{kaplan@math.huji.ac.il}

\author{Tomasz Rzepecki}
\grant{The second author is supported by the Narodowe Centrum Nauki grant no.\ 2016/22/E/ST1/00450 and the Lady Davis fellowship.}
\address{The Hebrew University of Jerusalem and Uniwersytet Wrocławski}
\orcid{https://orcid.org/0000-0001-9786-1648}
\email{tomasz.rzepecki@math.uni.wroc.pl}

\author{Daoud Siniora}
\address{}
\orcid{}
\email{daoud.siniora@gmail.com}

\date{}

\begin{document}
	\maketitle
	
	\begin{abstract}
		We show that the countable universal homogeneous meet-tree has a generic automorphism, but it does not have a generic pair of automorphisms.
	\end{abstract}
	
	\section{Introduction}
	A countable structure $M$ has a \emph{generic automorphism} if its automorphism group has a comeagre conjugacy class. This is an important property which has certain implications on the automorphism group $G$ (for example: every element is a commutator, $G$ cannot be written as a proper free product with amalgamation and more; see \cite[Proposition 4.2.12]{MacSurvey}).
	
	A much stronger property is having \emph{ample generics}, which means that for all $n\in \mathbf N$, $ G $ has a generic tuple of length $ n $ or, in other words, $ G^{n} $ contains a comeagre orbit under the action of $ G $ by diagonal conjugation. Having ample generics implies in particular the small index property \cite[Theorem 5.2.5]{MacSurvey} --- which is desirable because, for instance, it implies that if $M$ is $\aleph_0$-categorical, then the automorphism group $ G $ (as a pure group) determines $ M $ up to bi-interpretability (among all countable $\aleph_0$-categorical structures; see e.g.\ \cite[Proposition 5.2.2]{MacSurvey}).
	
	This article started as an attempt to find a counterexample to a question of Dugald Macpherson (also appearing in the third author's thesis \cite[Chapter 7, Question 10]{Sin17}) which asks whether there is an ultrahomogeneous $ \aleph_0 $-categorical structure admitting ample generics with the strict order property (i.e., defining a partial order with infinite chains). Note that the answer to this question in general is negative (since \cite{Kwi12} shows that the countable atomless Boolean algebra has ample generics), so it seems that the question should be investigated under the NIP assumption. The obvious candidate, $ (\bQ,<) $, fails: by the works of Hodkinson (unpublished), Truss \cite{TrussNoGenericPair}, and the third author (who gave a new proof of this result) \cite[Lemma 6.1.1]{Sin17}, we know that $ \Aut(\bQ,<) $ has no generic pair of automorphisms (although $ \Aut(\bQ,<) $ does have a generic automorphism by \cite{Tru92,KT01}).
	
	Meet-trees are partial orders in which for every element $ a $, the set of elements below it is linearly ordered, equipped with a meet function (see Definition~\ref{definition:trees}). They are often used as a basis for interesting examples in the realm of NIP unstable theories (in the sense of Shelah's classification theory). For example, in \cite{946,975}, trees were the basis for a counterexample to an old conjecture of Shelah regarding the existence of indiscernibles in NIP. See also \cite[][Section 2.3.1]{SimGuide}. Finite meet-trees form a Fra\"{i}ss\'{e} class and thus there is a countable universal ultrahomogeneous and $ \aleph_0 $-categorical meet-tree $ \bT $ which we call the \emph{universal countable dense meet-tree}. We therefore thought that it would be interesting to understand the automorphism group of $ \bT $ with respect to generic automorphisms.
	
	Our main result is:
	\begin{mainthm}
		\phantomsection
		\label{mainthm}
		Let $ \bT $ be the universal countable dense meet-tree. Then its automorphism group $ G=\Aut(\bT) $ has a generic automorphism, but not a generic pair of automorphisms.
		
		(See Corollary~\ref{corollary:no_generic_pair} and Theorem~\ref{theorem:gen_automorphism_exists}.)\xqed{\lozenge}
	\end{mainthm}
	
	Note that the \hyperref[mainthm]{Main Theorem} implies that in particular, $\Aut(\bT)$ does not have ample generics --- which, as was mentioned, would imply having the small index property. However, it follows from \cite[Theorem 4.1]{DHM89} that $\Aut(\bT)$ (as well as the analogues for meet-trees of bounded arity, for which we also show the non-existence of a generic pair) does have the small index property.
	
	We use a criterion for having a generic automorphism that was established by Truss \cite[Theorem 2.1]{Tru92} and then improved to a characterisation independently by Ivanov \cite[Theorem 1.2]{Ivanov} and Kechris and Rosendal \cite[Theorem 6.2]{KR07}. Namely, to show that an ultrahomogeneous structure $ M $ with age $ \cK $ has a generic $ n $-tuple of automorphisms, one needs to prove that the class $ \cK^{n} $ of pairs $ (A,\bar{p}) $ such that $ A\in \cK $ and $ \bar{p} $ is an $n$-tuple of partial automorphisms of $ A $, has the joint embedding property (JEP) and a version of the amalgamation property which we call the existential amalgamation property (EAP)\footnote{In \cite{KR07} and \cite{Ivanov} this property was called the ``weak amalgamation property'' (WAP), and ``almost amalgamation property'', respectively, but we chose this term, coined in an unpublished work of Ben-Yaacov and Melleray, as it is more descriptive.} (see Fact~\ref{fact:amalgam_and_generics}).
	
	To show EAP in the case of finite meet-trees, we find a cofinal subclass of $ \cK^1 $ consisting of amalgamation bases (which is the aforementioned criterion from \cite{Tru92}). Starting with some member $ (A,p_A) $, we extend $ p_A $ to a partial automorphism $ p_B $ on a bigger domain $ B $ in such a way that $(B,p_B)$ is an amalgamation base in $ \cK^1 $. The idea in finding $ p_B $ is model-theoretic: instead of giving a precise description of $ p_B $,  we define it as being ``pseudo existentially closed'' in the sense that, roughly, any behaviour that happens in some extension of $ p_B $ is already witnessed in $ p_B $ itself (see Definition~\ref{definition:new_strong}). This method eliminates the need for a careful analysis of the interactions between multiple orbits of $p_B$ that we might otherwise need.
	
	Finally let us remark that while we did not check all the details, our methods seem to recover a proof of existence of generics for $(\mathbf Q,<)$ (see Remark~\ref{remark:works_for_linear}), and also provide a similar result for lexicographically ordered meet-trees (see Remark~\ref{remark:works_for_lex_trees}). They could also be helpful in finding generics in the case of meet-trees of bounded arity (see Remark~\ref{remark:works_for_bounded_arity}; see also Corollary~\ref{corollary:no_generic_pair} for the non-existence of generic pairs in this case).
	
	The paper is organized as follows. In Section~\ref{section:preliminaries} we recall Fraïssé classes and formally define EAP. We then discuss trees and give their basic properties. In Section~\ref{section:There is no generic pair} we prove that the automorphism group of the countable dense tree does not have a generic pair of automorphisms (in fact we give a more general statement, also about trees with bounded arity). In Section~\ref{section:determined partial automorphisms} we follow \cite{KT01} and discuss \emph{determined} finite partial automorphisms in an abstract context, giving a sufficient condition under which they form a class of amalgamation bases. In Section~\ref{section:Orbits in meet-trees} we discuss the possible orbits of partial automorphisms in meet-trees. Finally, in Section~\ref{section:Finding determined automorphisms} we prove that the class of determined partial automorphisms of meet-trees is indeed cofinal, thus proving EAP for the class of finite meet-trees.
	
	We end the introduction with some open questions.
	
	Our main results are similar to the ones in \cite{KT01,TrussNoGenericPair} regarding $ (\bQ,<) $ and the universal partial order (although it is not known if the universal partial order has a generic pair); more recently, a new preprint \cite{KM19} appeared giving similar results on two different structures (the universal ordered boron tree --- roughly speaking, a graph theoretic binary tree with a lexicographical order --- and the universal ordered poset). The latter's motivation came from a different yet related question of finding an ultrahomogeneous \emph{ordered} structure whose automorphism group has ample generics and is extremely amenable (in other words, by \cite{KPR}, its age has the Ramsey property). 
	
	Even more recently, a preprint \cite{Duc19} appeared, which analyses the homeomorphism groups of the so-called Ważewski dendrites, in particular showing the analogue of the \hyperref[mainthm]{Main Theorem} for the group of homeomorphisms of the Ważewski dendrite $D_{\{\infty \}}$ (into which $\Aut(\bT)$ naturally embeds as a meagre subgroup, namely as the stabiliser of an end point). We were not aware of this work. By private communication with Duchesne, it appears that though it is not written explicitly there, our result follows from the main theorem there and vice versa. Both papers use the aforementioned criterion for establishing the existence of a generic automorphism, but the methods of finding amalgamation bases are different. (Interestingly, from the point of view of \cite{Duc19}, meet-trees with bounded arity are analogues to dendrites with bounded branching. The latter do not have a dense conjugacy class according to \cite[Proposition 1.1]{Duc19}, while the former do, see also Remark \ref{remark:works_for_bounded_arity}.)
	
	\begin{question}
		What is the correct generalization of all these results?
	\end{question}
	
	Model theoretically, it seems appealing to consider the situation in general NIP $ \aleph_0 $-categorical structures with perhaps some further restrictions, as was done in \cite{PierreNIPomegacategoricalrank1}. This situation does not quite generalize ours since trees are not ``rank 1" in the sense defined there, and the universal partial order is not NIP.
	
	Another natural question that comes to mind is the following. In all the examples known to us, if EAP occurs for $ \cK^{1} $ then in fact there is a cofinal class of amalgamation bases (a property we denote by CAP, see Definition~\ref{definition:eap}).
	
	\begin{question}
		\label{question:EAP_not_CAP}
		Is it always the case that EAP is equivalent to CAP for the class $ \cK^{1} $ (where $ \cK $ is any Fra\"{i}ss\'{e} class)?
	\end{question}
	
	We are not sure what the situation is with trees of bounded arity; it may be a candidate for a counterexample --- see Remark \ref{remark:works_for_bounded_arity}.
	
	Since we know that the generic automorphism exists, it seems natural to ask how complicated it is.
	\begin{question}
		What can be said, model theoretically --- in terms of classification theory --- about the structure $ (\bT,\sigma) $, where $ \sigma $ is a generic automorphism? 
	\end{question}
	
	\subsubsection*{Acknowledgement}
	We would like to thank Bruno Duchesne for letting us know of his work mentioned above and for his other comments. We would also like to thank the anonymous referee for their helpful remarks.

	\section{Preliminaries}
	\label{section:preliminaries}
	\subsection*{Fraïssé classes and limits}
	We briefly recall the basic notions related to Fraïssé classes. See \cite[Chapter 6]{Hod} for more exposition. In contrast to \cite{Hod}, it will be convenient for us to consider classes of structures with partial functions. Formally, they can be thought of as relations (via their graphs), so this is only a superficial change.
	
	\begin{definition}
		Let $\cK$ be a class of first order structures (possibly with some partial functions), closed under isomorphisms.
		\begin{itemize}[nosep]
			\item
			We say that $\cK$ has the \emph{hereditary property} (\emph{HP}), if given any $A\in \cK$ and substructure $B\subseteq A$, we have $B\in \cK$.
			\item
			We say that $\cK$ has the \emph{joint embedding property} (\emph{JEP}) if given any $A_1,A_2\in \cK$, there is some $B\in \cK$ such that both $A_1$ and $A_2$ can be embedded in $B$.
			\item
			We say that $A\in \cK$ is an \emph{amalgamation base} (in $\cK$) if given any $B_1,B_2\in \cK$ and embeddings $i_j\colon A\to B_j$ for $j=1,2$, there is some $C\in \cK$ and embeddings $i'_j\colon B_j\to C$ for $j=1,2$ such that $i'_1\circ i_1=i'_2\circ i_2$.
			\item
			We say that $\cK$ has the \emph{amalgamation property} (\emph{AP}) if every $A\in \cK$ is an amalgamation base in $\cK$.
			\item
			We say that $\cK$ is \emph{uniformly locally finite} if given any $n\in \mathbf N$, there is an upper bound on the size of an $n$-generated element of $\cK$.
			\item
			A \emph{Fraïssé class} is is a class of finitely generated structures which is closed under isomorphisms, and has HP, JEP and AP.
			\item
			The \emph{age} $\Age(M)$ of a first-order structure $M$ is the class of all (isomorphism types of) finitely-generated substructures of $M$.\xqed{\lozenge}
		\end{itemize}
	\end{definition}
	We recall the notion of a partial automorphism (which is fundamental for this paper).
	\begin{definition}
		Given a first order structure $M$, a partial function $M\to M$ is called a \emph{partial automorphism} if it preserves the quantifier-free types over the empty set.
	\end{definition}
	
	\begin{fact}
		\label{fact:fraisse_theorem}
		If $\cK$ is a Fraïssé class, then there is a unique (up to isomorphism), countable structure $\mathbf K$ whose age is exactly $\cK$ and which is ultrahomogeneous (i.e.\ every finite partial automorphism of $\mathbf{K}$ extends to an automorphism).
		
		Furthermore, if $\cK$ is uniformly locally finite, then the theory of $\mathbf K$ has quantifier elimination and (if $\mathbf K$ is infinite) is $\aleph_0$-categorical.
	\end{fact}
	\begin{proof}
		This is classical; see for instance \cite[Theorem 6.1.2]{Hod} and \cite[Theorem 6.4.1]{Hod}.
	\end{proof}
	
	\begin{definition}
		The structure $\mathbf K$ as in Fact~\ref{fact:fraisse_theorem} is called the \emph{(Fraïssé) {limit}} of $\cK$.\xqed{\lozenge}
	\end{definition}
	\subsection*{Partial automorphisms, generic automorphisms}
	If $M$ is a countable first-order structure, then $\Aut(M)$ has a natural Polish group structure (with the pointwise convergence topology), and we can use descriptive set theory to study it. The notion of a generic is due to \cite{Tru92} and \cite{KR07} (for tuples).
	\begin{definition}
		\label{definition:generic}\leavevmode
		\begin{itemize}[nosep]
			\item
			An element $\sigma\in \Aut(M)$ is called \emph{generic} if its conjugacy class is comeagre in $\Aut(M)$ (i.e.\ it contains a dense $G_\delta$ set).
			\item
			More generally, a tuple $(\sigma_1,\ldots,\sigma_n)\in \Aut(M)$ is \emph{generic} if its diagonal conjugacy class (i.e.\ the orbit under the action $\Aut(M)$ on $\Aut(M)^n$ by coordinatewise conjugation) is comeagre.
			\item
			We say that $\Aut(M)$ has \emph{ample generics} if it has generic tuples of elements of arbitrary length.\xqed{\lozenge}
		\end{itemize}
	\end{definition}

	The definition of EAP below is due to \cite{Ivanov} (where it is called \emph{almost amalgamation property}). It is also used in \cite{KR07} (where it is called \emph{weak amalgamation property}).
	\begin{definition}Fix a class $\cK$ of first order structures, closed under isomorphisms.
		\label{definition:eap}
		\begin{itemize}[nosep]
			\item
			We say that $\cK$ has \emph{EAP} (\emph{existential amalgamation property}) if for every $A\in \cK$, there is some $B\in \cK$ and an embedding $i_{AB}\colon A\to B$ such that for any embeddings $i_{BC}\colon B\to C$, $i_{BD}\colon B\to D$  (with $C,D\in\cK$), there are embeddings $i_{CE}\colon C\to E$ and $i_{DE}\colon D\to E$ (where $E\in \cK$) such that $i_{CE}\circ i_{BC}\circ i_{AB}=i_{DE}\circ i_{BD}\circ i_{AB}$.
			\item
			We say that $\cK$ has \emph{CAP} (\emph{cofinal amalgamation property}) if we can choose the $B$ such that $i_{CE}\circ i_{BC}=i_{DE}\circ i_{BD}$ (i.e.\ if $\cK$ has a cofinal subclass of amalgamation bases).\xqed{\lozenge}
		\end{itemize}
	\end{definition}

	\begin{remark}
		\label{remark:cap_eap}
		If $\cK$ has CAP, it has EAP.\xqed{\lozenge}
	\end{remark}
	
	\begin{definition}
		Suppose that $\cK$ is a Fraïssé class. Let $\cK^n$ be the class of pairs $(A,\bar{p})$ where $A\in \cK$ and  $\bar{p}=(p_i)_{i<n}$ is an $n$-tuple of partial automorphisms of $A$.\footnote{\cite{KR07} denoted this class by $\cK^n_{p}$, but we decided to omit the $p$ from the notation to ease the notational burden a bit.}
	\end{definition}
	
	\begin{fact}
		\label{fact:amalgam_and_generics}
		Fix any $n\in \mathbf N$ and $\cK$ is a Fraïssé class with limit $\mathbf K$, then the following are equivalent:
		\begin{itemize}
			\item
			$\Aut(\mathbf K)$ has a generic $n$-tuple,
			\item
			$\cK^n$ has JEP and EAP.
		\end{itemize}
	\end{fact}
	\begin{proof}
		This is \cite[Theorem 6.2]{KR07}. Under the additional assumption that $\mathbf K$ is $\aleph_0$-categorical (which includes our applications), this is a special case of \cite[Theorem 1.2]{Ivanov}.
	\end{proof}
	
	\begin{corollary}
		\label{corollary:sufficient_for_generics}
		If $\cK$ is a Fraïssé class with limit $\mathbf K$ and $\cK^n$ has JEP and CAP, then $\mathbf K$ has a generic $n$-tuple of automorphisms.
	\end{corollary}
	\begin{proof}
		Immediate by Remark~\ref{remark:cap_eap} and Fact~\ref{fact:amalgam_and_generics}. (For $n=1$ this is essentially \cite[Theorem 2.1]{Tru92}.)
	\end{proof}

	\subsection*{Trees}
	\begin{definition}\leavevmode
		\label{definition:trees}
		\begin{itemize}[nosep]
			\item
			A \textit{tree} is a partially ordered set $(A, \leq)$ which is \emph{semilinear} (that is, for every $a_0\in A$, the set $A_{\leq a_0}=\{a\in A\mid a\leq a_0 \}$ is linearly ordered) and such that every pair of elements has a common lower bound.
			\item
			A \emph{meet-tree} (or $\meet$-tree) $(A, \leq, \meet)$ is a tree which is also a lower semilattice, i.e.\ a tree $(A,\leq)$ together with a binary (meet or infimum) function ${\meet}:A^2\to A$ such that for every $a,b \in A$, $a\meet b$ is the largest element of $A_{\leq a}\cap A_{\leq b}$.\xqed{\lozenge}
		\end{itemize}
	\end{definition}
	
	\begin{remark}
		If $(A,\leq)$ is a tree with the property that every pair has an infimum, then there is a unique way to expand it to a meet-tree. In particular, every finite tree has a unique meet-tree structure. However, not every embedding of finite trees yields an embedding of the resulting meet-trees.\xqed{\lozenge}
	\end{remark}

	\begin{remark}
		The $\meet$ operation is associative, commutative and idempotent.\xqed{\lozenge}
	\end{remark}
	
	\begin{definition}
		Given a tree $T$, the \emph{arity} of $T$ is the maximal size of a set $A\subseteq T$ of pairwise incomparable elements such that if $a_1,a_2,a_3\in A$ are distinct and $b\in T$ is such that $b<a_1$ and $b<a_2$, then $b<a_3$ (or $\infty$ if there is no finite bound).
	\end{definition}
	
	\begin{remark}
		If $T$ is a meet-tree, arity can be equivalently defined in the following way: we say that $T$ is $k$-ary if $k$ is the maximal size of a subset $A\subseteq T$ such that all pairs have the same meet, which is not equal to any of them. Note that this definition shows that being at most $k$-ary is an universal property in the language of meet-trees (it is not hard to see that it is not an universal property in the pure order language).\xqed{\lozenge}
	\end{remark}
	
	\begin{fact}
		\label{fact:generic_trees}
		The class of all finite meet-trees is a Fraïssé class (in the language of meet-trees). Given any positive integer $k$, the class of all finite $k$-ary meet-trees is a Fraïssé class (in the language of meet-trees).
		
		Consequently, there is a countable generic meet-tree, $\bT$, and for every $k$ there is a countable generic $k$-ary meet-tree, $\bT_k$. $\bT$ and each $\bT_k$ is $\aleph_0$-categorical, ultrahomogeneous and has elimination of quantifiers.
	\end{fact}
	\begin{proof}
		The first part is straightforward. The second part follows from Fact~\ref{fact:fraisse_theorem}.
	\end{proof}
	(Notice that in particular, a $1$-ary meet-tree is simply linear, and $\bT_1$ is (interdefinable with) the universal linear ordering, isomorphic to $(\mathbf Q,\leq)$.)
	
	We will use the notation $\bT$ and $\bT_k$ throughout the paper.
	
	\begin{remark}
		The class of all finite trees (in the pure order language) is \emph{not} a Fraïssé class --- it does not have the amalgamation property; it does, however have a model companion: there is a unique countable existentially closed tree (in the pure order language). It is an $\aleph_0$-categorical binary tree, but in contrast to $\bT_2$, which is a meet-tree, no two incomparable elements have a meet. (See \cite{BBPP18} for more details.) It might be interesting to ask whether it has a generic automorphism, but the methods based on Fact~\ref{fact:amalgam_and_generics} used in this paper do not seem to apply directly.\xqed{\lozenge}
	\end{remark}
	
	\begin{remark}
		Another approach to trees is by graph theory: we can identify every finite tree with an acyclic directed graph, but the Fraïssé limits of classes of finite trees in this language will be quite different (for instance, the ``order'' on each branch will not be dense). In \cite{KM19}, the authors study the existence of generic automorphisms in this context.\xqed{\lozenge}
	\end{remark}

	The following simple observation, reminiscent of the ultrametric triangle inequality, will be immensely useful in the rest of this paper.
	\begin{fact}\label{fact:lemmaxyz}
		Let $a,b,c$ be elements of a meet-tree $T$. Then:
		\begin{itemize}
			\item
			$a\meet b\meet c=a\meet b$ or $a\meet b\meet c=a\meet c$,
			\item
			if $a\meet b>a\meet c$, then $a\meet c=b\meet c$,
			\item
			if $a\meet b\geq a\meet c$, then $a\meet c\leq b\meet c$.
		\end{itemize}
	\end{fact}
	\begin{proof}
		The proof is straightforward (using semilinearity).
%
%
	\end{proof}
	Note that the first bullet in Fact~\ref{fact:lemmaxyz} says in particular that if $T$ is a meet-tree and $A\subseteq T$, then the substructure of $T$ generated by $A$ is exactly the set of meets of pairs of elements of $A$.
	
	\begin{definition}
		\label{definition:cut}
		Given a partial order $(P,\leq)$, and an $A\subseteq P$, 
		\begin{itemize}[nosep]
			\item
			a (\emph{lower}) \emph{cut} in $A$ is a downwards closed, (upwards) directed subset of $A$ (i.e.\ a $C\subseteq A$ such that for any $c\in C$ and $a\in A$, if $a\leq c$, then $a\in C$, and for any $c_1,c_2\in C$, there is some $c\in C$ such that $c\geq c_1,c_2$).
			\item 
			an \emph{order type} over $A$ is simply a complete quantifier-free type over $A$ in $P$, in the pure order language (with equality).\xqed{\lozenge}
		\end{itemize}
	\end{definition}
	
	\begin{remark}
		\label{remark:cuts_and_order_types}
		Fix a partial order $(P,\leq)$ (possibly with some additional structure) and a subset $A\subseteq P$.
		
		Given an order type $p$ over $A$, we have two corresponding cuts in $A$, namely $p_{\geq}:=\{a\in A\mid p\vdash x\geq a \}$ and $p_{>}:=\{a\in A\mid p\vdash x>a \}$. The two are equal if and only if $p$ is not realised in $A$. If $P$ is linear, then the two cuts uniquely determine $p$, but in general, it is not true.
		
		Conversely, given a \emph{nonempty} cut $C\subseteq A$, there is an order type $p$ over $A$ such that $p_{\geq}=C$.\xqed{\lozenge}
	\end{remark}
	
	\begin{definition}
		Given any $b\in P$, by the \emph{order type of $b$ over $A$}, $\otp(b/A)$, we mean simply the quantifier-free type of $b$ over $A$ in the order language, and by \emph{the cut of $b$ in $A$} we mean simply the cut $\{a\in A\mid b\geq a \}$ (which is the same as $\otp(b/A)_{\geq}$ in the notation of Remark~\ref{remark:cuts_and_order_types}).\xqed{\lozenge}
	\end{definition}

	\begin{remark}
		\label{remark:cuts_in_trees}
		In a tree, a directed set is linear, so a cut is simply a downwards closed chain.\xqed{\lozenge}
	\end{remark}
	
	\begin{remark}
		\label{remark:ordering_cuts}
		Note that for any poset $(P,\leq)$, the cuts in $P$ are ordered simply by inclusion. If we denote by $\hat P$ the set of all cuts in $P$, partially ordered by inclusion, it is easy to see that:
		\begin{itemize}[nosep]
			\item
			$P$ naturally embeds into $\hat P$ (where each element is identified with its cut in $P$);
			\item
			$\hat P$ is complete in the sense that every directed subset of $\hat P$ has a supremum (namely, the union);
			\item
			if $P$ is a tree, then so is $\hat P$, and moreover, $\hat P$ has a canonical meet-tree structure, with the meet given by intersection;
			\item
			if $(P,\leq,\meet)$ is a meet-tree, then $(P,\leq,\meet)$ is a substructure of $(\hat P,\subseteq,\cap)$.\xqed{\lozenge}
		\end{itemize}
	\end{remark}
	
	The following fact appears to be folklore.
	\begin{fact}\label{fact:typesintrees}
		Let $B$ be a meet-tree, $b\in B$, and $A\subseteq B$ a finite substructure (in the language of meet-trees). Put $b' = \max\{x \meet b : x\in A\}$, and let $a\in A$ be such that $a\geq b'$.
		Then the quantifier-free type $\qftp(b/A)$ is determined by knowing whether $b=b'$ or $b'<b$, the element $a$, and the order type of $b'$ over $A_{\leq a}$.
		
		(In particular, if $B$ has quantifier elimination (for instance, if $B=\bT$ or $B=\bT_k$ for some $k>0$, in the sense of Fact~\ref{fact:generic_trees}), this also determines $\tp(b/A)$.)
	\end{fact}
	\begin{proof}
		The proof is left as an exercise to the reader.
	\end{proof}

	\begin{corollary}
		\label{corollary:meet_above_prime}
		Fix a meet-tree $T$, and arbitrary $a_1,a_2\in T$, as well as a finite $B\subseteq T$. Then if $a_1\meet a_2>\max_{b\in B}b\meet a_1$, then $\qftp(a_1/B)=\qftp(a_2/B)$ (in particular, if $T=\bT$, then $\tp(a_1/B)=\tp(a_2/B)$).
	\end{corollary}
	\begin{proof}
		Note that the assumption immediately implies that $a_1,a_2\neq \max_{b\in B}b\meet a_1$, so by Fact~\ref{fact:typesintrees} it is enough to show that $\max_{b\in B}b\meet a_2=\max_{b\in B}b\meet a_1$. But the assumption tells us immediately that for every $b\in B$ we have $b\meet a_1<a_1\meet a_2$, so by Fact~\ref{fact:lemmaxyz}, $b\meet a_1=b\meet a_2$, which completes the proof.
	\end{proof}
	
	\section{There is no generic pair}
	\label{section:There is no generic pair}
	In this section, we will use Fact~\ref{fact:amalgam_and_generics} to show that the countable generic meet-trees $\bT$ and $\bT_k$ do not admit generic pairs of automorphisms.
	\begin{definition}\leavevmode
		\begin{itemize}[nosep]
			\item
			Given a pair $(g_1,g_2)$ of partial functions, we say that $(g_1',g_2')$ is an \emph{extension} of $(g_1,g_2)$ if $g_1',g_2'$ are partial functions such that $g_1\subseteq g_1'$ and $g_2\subseteq g_2'$.
			\item
			Given two pairs $(g_1',g_2')$, $(g_1'',g_2'')$ of partial automorphisms of a partially ordered set $P$ and a point $a\in P$, we say that they are \emph{irreconcilable over $a$} if there is no partial ordering $Q$ and partial automorphisms $f_1,f_2$ of $Q$ such that $(P,g_1',g_2',a),(P,g_1'',g_2'',a)\hookrightarrow (Q,f_1,f_2,b)$ for some $b\in Q$.\xqed{\lozenge}
		\end{itemize}
	\end{definition}
	
	\begin{remark}
		If $(g_1',g_2'), (g_1'',g_2''),(g_1''',g_2''')$ are pairs of partial automorphisms of a poset $P$ such that $(g_1',g_2')$ and $(g_1'',g_2'')$ are irreconcilable over $a\in P$, while $(g_1''',g_2''')$ extends $(g_1',g_2')$, then it is easy to see that:
		\begin{itemize}
			\item
			$(g_1''',g_2''')$ and $(g_1'',g_2'')$ are irreconcilable over $a$,
			\item
			there is no $f\in \Aut(P)$ fixing $a$, such that $g_1'\cup fg_1''f^{-1}$ and $g_2'\cup fg_2''f^{-1}$ are partial automorphisms.\xqed{\lozenge}
		\end{itemize}
	\end{remark}
	
	
	\begin{proposition}
		\label{proposition:automorphism_extends}
		Suppose $L$ is a dense linear order, unbounded from below. Suppose $g$ is a finite partial automorphism of $L$, $a,b\in L$ are such that $b\leq g(a)\leq a$. Then there are $d_1,d_2\in L$ such that $g\cup \{(d_1,b),(b,d_2)\}$ is a partial automorphism.
	\end{proposition}
	\begin{proof}
		We may assume without loss of generality that $L$ is countable (by replacing it with a countable dense subset containing $b$, the range and the domain of $g$). In this case, there is some $K\supseteq L$, countable, totally ordered without endpoints; since $L$ is dense without lower bound, we may assume that for every $k\in K\setminus L$, we have $k>a$.
		
		Countable dense orderings without endpoints are ultrahomogeneous, so $g$ can be extended to some $f\in \Aut(K)$. Put $d_1:=f^{-1}(b),d_2:=f(b)$. Now, we have $b\leq a$, so $d_2=f(b)\leq f(a)=g(a)\leq a$, so $d_2\in L$. On the other hand, $b\leq g(a)=f(a)$, so $d_1=f^{-1}(b)\leq f^{-1}(f(a))=a$, so $d_1\in L$. Since $f$ is a partial automorphism, so is $g\cup \{(d_1,b),(b,d_2)\}\subseteq f$, so we are done.
	\end{proof}
	The following fact is essentially \cite[Lemma 6.1.1]{Sin17}, but we slightly strengthen the conclusion using Proposition~\ref{proposition:automorphism_extends}, and we give a more detailed proof.
	\begin{proposition}\label{proposition:daoudthesis}
		Let $(L,<)$ be a dense linear order, unbounded from below, and take some $a\in L$. Let $q_1,q_2$ be finite partial automorphisms of $L$ such that $q_1(a),q_2(a)<a$. Then $(q_1,q_2)$ admits two irreconcilable extensions $(q_1',q_2')$ and $(q_1'',q_2'')$ to pairs of finite partial automorphisms of $L$.
	\end{proposition}
	\begin{proof}
		Given a finite partial automorphism $g$ of $L$ and an element $b_0\in L$, put $C(b_0,g):=\#\{b\in\dom(g)\mid b\leq b_0 \}+ \#\{b\in\range(g)\mid b\leq b_0 \}$. Note that $C(b_0,g)$ is always a non-negative integer; note also that $C$ is monotone in $b_0$: if $b_0'\leq b_0$, then $C(b_0',g)\leq C(b_0,g)$. Given a pair $(g_1,g_2)$ of finite partial automorphisms, denote by $m(g_1,g_2)$ the minimal element of the $\langle g_1,g_2\rangle$-orbit of $a$ (i.e.\ the smallest element that can be obtained from $a$ by successive applications of $g_1,g_2,g_1^{-1},g_2^{-1}$). Write $C_m(g_1,g_2)$ for $C(m(g_1,g_2),g_1)+C(m(g_1,g_2),g_2)$.
		
		Call a pair $(g_1,g_2)$ \emph{minimal} if $C_m(g_1,g_2)$ is minimal among its extensions. We may assume without loss of generality that $(q_1,q_2)$ is minimal (otherwise, we can simply extend it).
		
		Write $c$ for $m(q_1,q_2)$ and $B$ for the union of domains and ranges of $q_1$ and $q_2$. Clearly, $c\in B$. Furthermore, $c\leq q_1(a),q_2(a)$, so it satisfies the assumptions of Proposition~\ref{proposition:automorphism_extends} for both $q_1$ and $q_2$.
		
		We claim that $c$ is in only one of $\dom(q_1),\dom(q_2),\range(q_1),\range(q_2)$. We will show that $c\notin \range(q_1)\cap \range(q_2)$ (the other cases are either analogous or easy to see). Suppose this is not the case. Since $c$ is minimal in its orbit, either $q_1^{-1}(c)>c$ or $q_2^{-1}(c)>c$ (otherwise, if $q_1(c)=q_2(c)=c$, the orbit of $c$ would have only one point, but we know that $a>c$ is in the same orbit). Suppose without loss of generality that the former holds. Let $d\in L$ be such that $g_1:=q_1\cup\{(c,d)\}$ is a partial automorphism (which we have by Proposition~\ref{proposition:automorphism_extends}). Then, since $c<q_1^{-1}(c)$, we have $d<c$. It is easy to see that $C(d,q_2)\leq C(c,q_2)-1$ (we do not count $c$ on the left-hand side) and
		\[C(d,g_1)\leq C(d,q_1)+1\leq C(c,q_1)-1+1=C(c,q_1).\]
		(The first inequality is because the only extra element we count in $C(d,g_1)$ is $d$, the second one is because $c$ is counted in $C(c,q_1)$, but not in $C(d,q_1)$.) Since $m(g_1,q_2)\leq d$, we conclude that $C_m(g_1,q_2)<C(c,q_1)+C(c,q_2)$, contradicting minimality of $(q_1,q_2)$.
		
		From now, we consider the case when $c\in \range(q_1)$ (the other cases are analogous), whence $c\notin \dom(q_2)\cup\range(q_2)$.
		
		Let $c^+>c$ be such that $(c,c^+]\cap B=\emptyset$, and let $c^-<c$ be such that $[c^-,c)\cap B=\emptyset$. Note that $c^+$ exists by density and the fact that $a>c$, while $c^-$ exists by density and the assumption that $L$ has no lower bound. We claim that $q_2':=q_2\cup \{(c,c^+)\}$ and $q_2'':=q_2\cup\{(c,c^-)\}$ are both partial automorphisms of $L$. Then if we put $q_1'=q_1''=q_1$, then clearly $(q_1',q_2')$ and $(q_1'',q_2'')$ will be irreconcilable over $a$. Note that since $c\notin \dom(q_2)$, we already know that $q_2'$ and $q_2''$ are well-defined partial functions, and by choice of $c^-$ and $c^+$, they are both injective, so it is enough to show that they preserve the (non-strict) order.
		
		Take any $b\geq c$, $b\in \dom(q_2)$; note that this implies that $b>c$, since $c\notin \dom(q_2)$. We need to show that $q_2(b)\geq c^+,c^-$. We know that $q_2(b)\neq c$ (because $c\notin\range(q_2)$) and of course $q_2(b)\in B$, so we have $q_2(b)\notin [c^-,c^+]$. It follows that it is enough to show that $q_2(b)\geq c^-$. Suppose towards contradiction that $q_2(b)<c^-$. Let $d\in L$ be such that $g_2:=q_2\cup \{(c,d)\}$ is a partial automorphism of $L$ (which we have by Proposition~\ref{proposition:automorphism_extends}). Since we have $c < b$, we have $d < q_2(b)<c^-<c$, so $d<c$. As before, we have $C(d,q_1)\leq C(c,q_1)-1$ and $C(d,g_2)\leq C(d,q_2)+1\leq C(c,q_2)-1+1$ (for the last inequality, note that $q_2(b)$ is counted in $C(c,q_2)$, but not in $C(d,q_2)$). Clearly, $m(q_1,g_2)\leq d$, so we have $C_m(q_1,g_2)\leq C(d,q_1)+C(d,g_2)<C(c,q_1)+C(c,q_2)$, a contradiction.
		
		Now, given some $b\leq c$ with $b\in \dom(q_2)$, we need to show that $q_2(b)\leq c^+,c^-$. As before, it is enough to show that $q_2(b)\leq c^+$. Arguing by contradiction, suppose $q_2(b)>c^+$. Using Proposition~\ref{proposition:automorphism_extends}, we can find $d$ such that $g_2:=q_2\cup\{(d,c)\}$ is a partial automorphism. But then, since $q_2(b)>c^+>c$, we have that $c>b=g_2^{-1}(q_2(b))>g_2^{-1}(c)=d$. Arguing as in the preceding paragraph, we contradict the minimality of $(q_1,q_2)$.
	\end{proof}
	
	\begin{lemma}
		\label{lemma:linear_union}
		Let $T$ be a meet-tree. Suppose $f,g$ are partial automorphisms of $T$ such that $\dom(f)$ and $\dom(g)$ are closed under $\meet$, and satisfy the condition that for every $\eta\in \dom(f)$, there is some $a_\eta\in \dom(g)$ such that $a_\eta\geq \eta$ and $g\restriction_{T_{\leq a_\eta}}\subseteq f$.
		
		Then $f\cup g$ is a partial automorphism of $T$ (whose domain is closed under $\meet$).
	\end{lemma}
	
	\begin{proof}
		Put $h:=f\cup g$. Note that if $\eta\in \dom(f)\cap \dom(g)$, then $g\restriction_{T_{\leq a_\eta}}\subseteq f$, so in particular, $f(\eta)=g(\eta)$, so $h$ is a well-defined partial function. Note also that given $\eta\in \dom(f)$, we have $h\restriction_{T_{\leq a_\eta}}\subseteq f$.
		
		\begin{claim*}
			The domain $\dom(h)$ is closed under $\meet$.
		\end{claim*}
		\begin{clmproof}
			Since $\dom(f)$ and $\dom(g)$ are closed under $\meet$, it is enough to show that if $\eta\in \dom(f)$ and $\nu\in \dom(g)$, then $\eta\meet \nu\in \dom(h)$.
			
			If $\eta\meet\nu=\eta$, then there is nothing to prove. Otherwise, $\eta\leq a_\eta\in \dom(g)$, so $a_\eta\meet\eta=\eta>\eta\meet \nu$. By Fact~\ref{fact:lemmaxyz}, we infer that $\eta\meet \nu=a_\eta\meet \nu$; since $a_\eta\in \dom(g)$, the conclusion follows.
		\end{clmproof}
		We will show that $h$ is an isomorphism between substructures of $T$. To that end, we need to show that $h$ is injective and that for any $\eta,\nu\in\dom(h)$, we have $\eta\leq \nu$ if and only if $h(\eta)\leq h(\nu)$, and that $h(\eta\meet\nu)=h(\eta)\meet h(\nu)$.
		
		Notice that $f^{-1}$ and $g^{-1}$ satisfy the assumptions of the proposition we are proving: indeed, given $\eta'=f(\eta)\in \dom(f^{-1})$, we have some $a_\eta\in \dom(g)$ with $a_\eta\geq \eta$ and $g\restriction_{T_{\leq a_\eta}}\subseteq f$. But this clearly implies that $f(a_\eta)=g(a_\eta)\geq\eta'$ and $g^{-1}\restriction_{T_{\leq f(a_\eta)}}\subseteq f^{-1}$. Thus by the first paragraph of this proof, $h^{-1}=f^{-1}\cup g^{-1}$ is a well-defined partial function, so $h$ is injective. By the same token, it is enough to show that $h$ preserves meets and that $\eta\leq\nu$ implies $h(\eta)\leq h(\nu)$ (i.e.\ the converse will follow).
		
		Now, let us fix arbitrary $\eta,\nu\in \dom(h)$ such that $\eta\leq \nu$. We need to show that $h(\eta)\leq h(\nu)$. If $\nu\in \dom(f)$, then it easily follows that $\eta\in \dom(f)$, so it is enough to consider the case when $\eta\in \dom(f)$ and $\nu\in \dom(g)$. It follows that $a_\eta\meet\nu\in \dom(g)$, so in fact $a_\eta\meet\nu\in \dom(f)\cap\dom(g)$ (because $a_\eta\meet \nu\leq a_\eta$); since trivially $a_\eta\meet\nu\geq \eta$, we have:
		\[
		h(\nu)=g(\nu)\geq g(a_\eta\meet \nu)=h(a_\eta\meet \nu)=f(a_\eta\meet \nu)\geq f(\eta)=h(\eta).
		\]
		
		Now, we need to show that $h(\eta\meet \nu)=h(\eta)\meet h(\nu)$. We may assume that $\eta\in \dom(f)$, $\nu\in\dom(g)$, and also that $\eta>\eta\meet\nu$ (if we have equality, then $\eta\leq \nu$, and the conclusion follows from the preceding paragraph). Then $\eta=a_\eta\meet\eta>\eta\meet\nu$, so (by Fact~\ref{fact:lemmaxyz}) $\eta\meet\nu=a_\eta\meet\nu$. As in the preceding paragraph, we have that $a_\eta\meet\nu\in \dom(f)\cap \dom(g)$. Thus, we have $f(a_\eta)\meet f(\eta)=f(\eta)>f(\eta\meet\nu)$. On the other hand $g(\eta\meet\nu)=g(a_\eta\meet\nu)=g(a_\eta)\meet g(\nu)$. In conclusion, we have $h(a_\eta)\meet h(\eta)>h(\eta\meet\nu)=h(a_\eta)\meet h(\nu)$. It follows by Fact~\ref{fact:lemmaxyz} that $h(\eta\meet\nu)=h(a_\eta)\meet h(\nu)=h(\eta)\meet h(\nu)$, so we are done.
	\end{proof}
	
	\begin{corollary}
		\label{corollary:nonamalgamation}
		Suppose $p_1,p_2$ are finite partial automorphisms of a dense, unrooted meet-tree $M$, such that for some $a\in M$, we have $p_1(a)=p_2(a)<a$. Then $(p_1,p_2)$ admits extensions $(p_1',p_2')$ and $(p_1'',p_2'')$ which are irreconcilable over $a$.
	\end{corollary}
	\begin{proof}
		Note that since $M$ is dense and unrooted, $M_{\leq a}$ is a dense linear ordering, without lower bound.
		
		For $i=1,2$, write $q_i:=p_i\restriction_{M_{\leq a}}$, and let $q_i',q_i''$ be the extensions given by Proposition~\ref{proposition:daoudthesis} (to partial automorphisms of $M_{\leq a}$), and put $p_i':=p_i\cup q_i'$, $p_i'':=p_i\cup q_i''$. Then by Lemma~\ref{lemma:linear_union} (with $f=q_i'$ or $f=q_i''$, $g=p_i$ and $a_\eta:=a$ for all $\eta$), $p_i'$ and $p_i''$ are partial automorphisms, and clearly, $(p_1',p_2')$ and $(p_1'',p_2'')$ are irreconcilable over $a$.
	\end{proof}
	
	\begin{corollary}
		\label{corollary:no_EAP2}
		If $M$ is a dense and unrooted meet-tree and $\cK:=\Age(M)$, then $\cK^2$ (the class of $\cK$-structures with pairs of partial automorphisms) does not have the EAP (see Definition~\ref{definition:eap}).
	\end{corollary}
	\begin{proof}
		Since $M$ is unrooted, in particular, it contains two elements $a>b$ and $p_1^0=p_2^0=\{(a,b) \}$ is a partial automorphism. Then any extension $(p_1,p_2)$ of $(p_1^0,p_2^0)$ satisfies the hypothesis of Corollary~\ref{corollary:nonamalgamation}, so it admits two extensions irreconcilable over $a$. This clearly implies the failure of EAP.
	\end{proof}
	
	The following corollary is the second half of the \hyperref[mainthm]{Main Theorem} (the first half we will prove later in Theorem~\ref{theorem:gen_automorphism_exists}).
	\begin{corollary}
		\label{corollary:no_generic_pair}
		If $\mathbf K$ is one of $\bT$ or $\bT_k$ for $k>0$ (in particular, if it is the dense linear ordering), then $\mathbf K$ does not admit a generic pair of automorphisms.
	\end{corollary}
	\begin{proof}
		In each case, $\cK:=\Age(\mathbf K)$ is a Fraïssé class with limit $\mathbf K$. By Fact~\ref{fact:amalgam_and_generics} and Corollary~\ref{corollary:no_EAP2}, it follows that the limit of $\cK$ (i.e.\ $\mathbf K$) does not admit a generic pair of automorphisms.
	\end{proof}
	
	\section{Determined partial automorphisms}
	\label{section:determined partial automorphisms}
	We aim to show that the universal countable meet-tree admits a generic automorphism (even though, by Corollary~\ref{corollary:no_generic_pair}, we already know it does not admit a generic pair of automorphisms). Very broadly, the proof follows \cite{KT01}. More precisely, we will find a sufficient condition for a partial automorphism to be an amalgamation base in the class $\mathcal K^1$ (where $\mathcal K$ is the class of finite meet-trees), and in the next section, we will find a cofinal class of automorphisms satisfying this condition, thus showing CAP for $\mathcal K^1$. This, in conjunction with Corollary~\ref{corollary:sufficient_for_generics}, will give us the existence of generics.
	
	\subsection*{Determined partial automorphisms in an abstract context}
	The notion of a strict extension and a determined automorphism is due to \cite{KT01}. We have slightly modified it: the authors of \cite{KT01} do not ask that the domain of a strict extension is a substructure, which is a trivial requirement in the case of relational structures. We also introduce the notion of a strictly positive extension.	
	\begin{definition}
		\label{definition:determined_KT}
		Let $M$ be a first order structure, and let $p$ be a finite partial automorphism of $M$.
		\begin{itemize}[nosep]
			\item
			We say that an extension $f\supseteq p$ of partial automorphisms of $M$ is \emph{strict} if it is an automorphism of a substructure (i.e.\ $\dom(f)=\range(f)$ is a substructure of $M$) and $\dom(f)$ is generated by the $f$-orbits of elements of $\dom(p)$.
			\item
			We say that an extension $f\supseteq p$ of partial automorphisms of $M$ is \emph{positively strict} if $f$ is an endomorphism of a substructure (i.e.\ $\dom(f)\supseteq \range(f)$ and $\dom(f)$ is a substructure of $M$) and $\dom(f)$ is generated by the \emph{positive} $f$-orbits of elements of $\dom(p)$ (i.e.\ the images of $\dom(p)$ by the positive powers of $f$).
			\item
			Given two extensions $f_1,f_2$ of $p$, we say that $f_1$ and $f_2$ are \emph{isomorphic over $p$} if there is an isomorphism $\theta\colon \dom(f_1)\to \dom(f_2)$ fixing $\dom(p)$ pointwise, such that $\theta\circ f_1=f_2\circ \theta$ (note that we do \emph{not} require that $\theta$ extends to an automorphism of $M$).
			\item
			We say that $p$ is \emph{[positively] determined} if, up to isomorphism over $p$, it admits a unique [positively] strict extension.\xqed{\lozenge}
		\end{itemize}
	\end{definition}
	Recall that we want to find a cofinal class of amalgamation bases in $\cK^1$. Proposition~\ref{proposition:amalgamation_implies_determined} and Lemma~\ref{lemma:determined_is_amalgamation_base} show that under reasonable assumptions (which, as we will see in Lemma~\ref{lemma:AP_in_trees_with_automorphisms}, are satisfied by the class of meet-trees), the notions of a determined automorphism and an amalgamation base in the class $\cK^1$ are essentially equivalent.
	\begin{proposition}
		\label{proposition:amalgamation_implies_determined}
		Suppose $\cK$ is a Fraïssé class with Fraïssé limit $\mathbf K$, and suppose $p$ is a partial automorphism of $\mathbf K$ such that for the structure $B\leq \mathbf K$ generated by $\dom(p)\cup\range(p)$, we have that $(B,p)$ is an amalgamation base in $\cK^1$. Then $p$ is determined.
	\end{proposition}
	\begin{proof}
		This is not important for our applications, so we only sketch the proof. We can do it by contraposition. If $f_1,f_2\supseteq p$ are strict, not isomorphic over $p$, then there are some finite $f_i'\subseteq f_i$ ($i=1,2$) such that for no automorphism $\theta$ of $\mathbf K$ fixing $\dom(p)$ we have that $f_1'\cup \theta f_2'\theta^{-1}$ is a partial automorphism of $\mathbf K$. This easily implies that for the appropriate $C_1,C_2$, we cannot amalgamate $(C_1,f_1')$ and $(C_2,f_2')$ over $(B,p)$.
	\end{proof}
	\begin{lemma}
		\label{lemma:determined_is_amalgamation_base}
		Suppose $\cK$ is a Fraïssé class with Fraïssé limit $\mathbf K$. Consider the class $\overline{\cK}$ of structures (possibly not finitely generated) whose age is contained in $\cK$, and let $\overline{\cK}^1_\totalaut$ be the class of $\overline{\cK}$-structures equipped with a (total) automorphism.
		
		If $\overline{\cK}^1_\totalaut$ has the AP, then for every determined partial automorphism $p$ of $\mathbf K$, we have that if $B\subseteq \mathbf K$ is generated by $\dom(p)\cup \range(p)$, then $(B,p)$ is an amalgamation base in $\cK^1$.
	\end{lemma}
	\begin{proof}
		Consider two embeddings (in $\cK^1$) $(B,p)\to (C_1,h_1),(C_2,h_2)$. We may assume without loss of generality that the embeddings are simply inclusions and $C_1,C_2\subseteq \mathbf K$. Let $f_1,f_2$ be strict extensions of $h_1,h_2$ (respectively) to partial automorphisms of $\mathbf K$. Then for $i=1,2$, let $f_i'\supseteq p$ be the restriction of $f_i$ to substructure of $\mathbf K$ generated by the $f_i$-orbit of $\dom(p)$, and put $B_i:=\dom(f_i')$ and $\overline C_i:= \dom(f_i)$ for $i=1,2$. Note that $f_i'\supseteq p$ is a strict extension for $i=1,2$, and since $p$ is determined, we have an isomorphism $\theta_2\colon (B_1,f_1')\to (B_2,f_2')$ fixing $\dom(p)$ (and hence $B$) pointwise. Write $(\overline B,f):=(B_1,f_1')$ and write $\theta_1$ for the inclusion mapping $\overline B\to \overline C_1$.
		
		Note that clearly, $\overline B,\overline C_1,\overline C_2\in \overline{\cK}$, and $\theta_i$ yields an embedding $(\overline{B},f)\to (\overline C_i,f_i)$ so, since $(\overline{B},f)$ is an amalgamation base in $\overline{\cK}^1_\totalaut$, we have some $(D,h)\in \overline{\cK}^1_\totalaut$ and embeddings $j_i\colon (\overline C_i,f_i)\to (D,h)$ (for $i=1,2$) such that $j_1\circ \theta_1=j_2\circ \theta_2$. Now, put $D_0':=j_1[C_1]\cup j_2[C_2]$ and let $D'$ be the substructure of $D$ generated by $D_0'\cup h[D_0']$. Put $h':=h\restriction_{D_0'}$, $j_i':=j_i\restriction_{C_i}$. The following diagrams summarise the maps constructed, with the maps on right hand side being restrictions of the maps on the left hand side.
		\begin{center}
			\begin{tikzcd}
				& (D,h) & && (D',h') &  \\
				(\overline{C_1},f_1)\ar[ur,"j_1"] && (\overline{C_2},f_2)\ar[ul,"j_2",swap] & ({C_1},h_1)\ar[ur,"j_1'"] && ({C_2},h_2)\ar[ul,"j_2'",swap] \\
				& (\overline B, f)\ar[ul,"\theta_1"]\ar[ur,"\theta_2",swap] & & & (B, p)\ar[ul,"\subseteq"]\ar[ur,"\subseteq",swap] &
			\end{tikzcd}
		\end{center}
		
		Since $D\in \overline{\cK}$, we have $D'\in \cK$, so clearly $(D',h')\in \cK^1$. Clearly, $j_i'$ is an embedding of $(C_i,h_i)$ into $(D',h')$, and since $j_1\circ \theta_1=j_2\circ \theta_2$ and each $\theta_i$ fixes $B$ pointwise, we also have
		\[
		j_1'\restriction_B=j_1\restriction_B=(j_1\circ \theta_1)\restriction_B=(j_2\circ \theta_2)\restriction_B=j_2\restriction_B=j_2'\restriction_B.
		\]
		Analogously, since $j_i\circ \theta_i\circ f\subseteq h\circ j_i\circ \theta_i$ (because $j_i\circ \theta_i$ is an embedding of $(\overline B,f)$ into $(D,h)$), we have $j_i'\restriction_B\circ p\subseteq h'\circ j_i'\restriction_{B}$, so 
		$(D',h')$ is an amalgam of $(C_1,h_1)$ and $(C_2,h_2)$ over $(B,p)$, which completes the proof.
	\end{proof}
	
	\begin{remark}
		The proof of Lemma~\ref{proposition:amalgamation_implies_determined} shows something slightly more general. Namely, if $p$ is a determined partial automorphism such that for the unique strict extension $\hat p\supseteq p$, the structure $(\dom(\hat p),\hat{p})$ is an amalgamation base in $\overline{\cK}^1_\totalaut$, then $(B,p)$ (defined as there) is an amalgamation base in $\cK^1$.\xqed{\lozenge}
	\end{remark}
	
	\begin{remark}
		\label{remark:AP_is_neccessary}
		The assumption that we have AP in $\overline\cK^1_\totalaut$ (or at least that $(\dom(\hat{p}),\hat p)$ is an amalgamation base there) is necessary in Lemma~\ref{lemma:determined_is_amalgamation_base}. See Remark~\ref{remark:determined_insufficient_bdd}.\xqed{\lozenge}
	\end{remark}
	
	\begin{remark}
		If $f\supseteq p$ is a strict extension, then $f_+$, its restriction to the substructure generated by the positive $f$-orbits of elements of $\dom(p)$, is a strictly positive extension. It is also easy to see that $\dom(f)=\bigcup f^{-n}\dom(f_+)$ (because it is a substructure of $\dom(f)$ and it contains the $f$-orbits of elements of $\dom(p)$).\xqed{\lozenge}
	\end{remark}
	
	We will be looking for determined partial automorphisms in order to apply Lemma~\ref{lemma:determined_is_amalgamation_base}. The following proposition shows that it is, in fact, enough to show positive determination.
	\begin{proposition}
		\label{proposition:positively_determined_is_determined}
		If $p$ is a positively determined finite partial automorphism, then it is determined.
	\end{proposition}
	\begin{proof}
		Fix two strict extensions $f,g\supseteq p$. Let $f_+,g_+$ be the respective positive parts. Fix an isomorphism $\theta_+\colon \dom(f_+)\to \dom(g_+)$ such that $g_+=\theta_+\circ f_+\circ \theta_+^{-1}$ (which exists by the assumption). Put $\theta_n:=g^{-n}\circ \theta_+\circ f^{n}$.
		
		Take any $a_+\in \dom(f_+)$. Then for any $k\geq 0$ we have:
		\[
		f^{k}(a_+)=f_+^{k}(a_+)=\theta_+^{-1}\circ g_+^{k}\circ \theta_+(a_+)=\theta_+^{-1}\circ g^{k}\circ \theta_+(a_+),
		\]
		and hence $g^{-k}\circ \theta_+\circ f^{k}(a_+)=\theta_+(a_+)$. It follows that if $n\leq m$, then $\theta_n\subseteq \theta_m$. Indeed, if $a\in \dom(\theta_n)$, then $f^n(a)\in \dom(f_+)$, so $g^{n-m}\circ \theta_+\circ f^{m-n}(f^n(a))=\theta_+\circ f^n(a)$, and thus
		\[
		\theta_m(a)=g^{-n}\circ g^{n-m}\circ\theta_+\circ f^{m-n}(f^n(a))=g^{-n}\circ \theta_+\circ f^n(a)=\theta_n(a).
		\]
		It follows that $\theta:=\bigcup_n \theta_n$ is a well-defined function. It is not hard to see that $\dom(\theta)=\bigcup_n\dom(\theta_n)=\dom(f)$ and $\range(\theta)=\bigcup_n(\range(\theta_n))=\dom(g)$. Finally, since $\theta_+$ is an isomorphism (between its domain and range), so is each $\theta_n$, and hence also $\theta$.
	\end{proof}
	
	The goal of this section is to find a useful criterion for a partial automorphism to be positively determined. To that end, we want to understand positively strict extensions of a partial automorphism by building it in a sequence of ``minimal'' steps. The following definition means to capture this notion of a single step.
	
	\begin{definition}
		\label{definition:immediate}
		Let $p\subsetneq f$ be partial automorphisms of a structure $M$. We say that $f$ is an \emph{immediate extension} of $p$ if $\dom(f)\setminus\dom(p)$ has only one element $a$, and moreover, $a\in \range(p)$ and the $p$-orbit of $a$ is the shortest non-cyclic orbit (i.e.\ no other non-cyclic orbit is shorter).\xqed{\lozenge}
	\end{definition}
	Note the extra condition at the end of the above definition might seem a bit technical. The point is that the infinite positively strict extensions are exactly the unions of infinite immediate chains, see Remark~\ref{remark:immediate_generate_positive}.
	
	\begin{remark}
		\label{remark:extension_to_substructure}
		If $p$ is a partial automorphism of a structure $M$, then $p$ has a unique extension $\overline{p}$ to a partial automorphism of $M$ such that $\dom(\overline{p})$ is the substructure of $M$ generated by $\dom(p)$. Furthermore, if $q\supseteq p$ is a partial automorphism extending $p$, then $\overline{q}\supseteq\overline{p}$.\xqed{\lozenge}
	\end{remark}
	
	\begin{remark}
		\label{remark:immediate_generate_positive}
		If $f\supseteq p$ is a positively strict extension, then by straightforward induction, there is a sequence $(f_n)_{n}$ (with $n$ ranging over $\omega$ or a finite ordinal if $f$ is finite) such that $f_0=p$, for each $n>0$, the extension $f_{n-1}\subseteq f_{n}$ is immediate, and $f=\bigcup_n \overline{f_n}$.
		
		Conversely, if $(f_n)_{n\in \omega}$ is such that $f_0=p$ and for each $n>0$, the extension $f_{n-1}\subseteq f_{n}$ is immediate, then $\bigcup_n \overline{f_n}$ is a positively strict extension of $p$.\xqed{\lozenge}
	\end{remark}
	The following Lemma will be useful in showing that the determination of a partial automorphism.
	\begin{lemma}
		\label{lemma:sufficient_for_determined}
		Let $M$ be a first order structure. Suppose $p$ is a finite partial automorphism of $M$ such that for every sequence $p=f_0\subsetneq f_1\subsetneq f_2\subsetneq\ldots\subsetneq f_n\subsetneq f_{n+1}$ such that for all $i\leq n$, the extension $f_i\subsetneq f_{i+1}$ is immediate, the following holds:
		
		\begin{itemize}[nosep]
			\item
			If $g_n\supseteq f_n$ is positively strict, then there is an automorphism $\tau_n\in \Aut(M)$ fixing $\dom(f_n)$ pointwise, such that $\tau_n\circ f_{n+1}\circ\tau_n^{-1}\subseteq g_n$.
		\end{itemize}
		
		Then $p$ is positively determined (and hence determined).
	\end{lemma}
	\begin{proof}
		Fix any two positively strict extensions $f,g\supseteq p$. For simplicity, suppose that $f$ is infinite (the case when $f$ is finite is analogous).
		
		Let $(f_n)_{n}$ be a sequence as in Remark~\ref{remark:immediate_generate_positive}, so that $f_0=p$ and $\bigcup_n \overline{f_n}=f$. We will recursively define a sequence ${\theta_n}$ of automorphisms of $M$ such that:
		\begin{itemize}
			\item
			for all $n$, the domain of $p$ is fixed pointwise by $\theta_n$,
			\item
			if $n\leq m$, then $\theta_n\restriction_{\dom(\overline{f_n})}=\theta_m\restriction_{\dom(\overline{f_n})}$,
			\item
			for all $n$ we have $\theta_n \circ \overline{f_n}\circ \theta_n^{-1}\subseteq g$.
		\end{itemize}
		Then $\theta:=\bigcup_n \theta_n\restriction_{\dom(\overline{f_n})}$ will clearly be a well-defined embedding $\dom(f)\to \dom(g)$, fixing $\dom(p)$, such that $\theta \circ f\circ \theta^{-1}\subseteq g$. It is easy to see that $\theta$ must be onto $\dom(g)$ (because its image is a substructure which contains the positive $g$-orbits of elements of $\dom(p)$), and hence $\theta\circ f\circ \theta^{-1}=g$. Since $f,g$ are arbitrary, it will follow that $p$ is positively determined, and hence (by Proposition~\ref{proposition:positively_determined_is_determined}) also determined.
		
		It is clear that $\theta_0=\operatorname{id}_M$ satisfies the all the conditions listed above. Suppose we have $\theta_0,\ldots,\theta_n$. Put $g_n:=\theta_n^{-1}\circ g\circ\theta_{n}$. Then $g_n\supseteq p$ is positively strict (because $g\supseteq p$ is positively strict and $\theta_n$ fixes $\dom(p)$ pointwise) and $g_n\supseteq f_n\supseteq p$, so $g_n$ is a positively strict extension of $f_n$. This allows us to take $\tau_n\in \Aut(M)$ as in the hypothesis. Put $\theta_{n+1}:=\theta_n\circ{\tau_n}$. Since $\tau_n$ fixes $\dom(f_n)$, it follows that $\theta_{n+1}\restriction_{\dom(f_n)}=\theta_n\restriction_{\dom(f_n)}$, and hence also $\theta_{n+1}\restriction_{\dom(\overline{f_n})}=\theta_n\restriction_{\dom(\overline{f_n})}$.
		
		Finally, since $\tau_n\circ f_{n+1}\circ \tau_{n}^{-1}\subseteq g_n=\theta_n^{-1}\circ g\circ\theta_{n}$, it easily follows that $\theta_{n+1}\circ \overline{f_{n+1}}\circ \theta_{n+1}^{-1}\subseteq g$.
	\end{proof}

	\subsection*{Determined partial automorphisms of trees}
	Now, we will proceed to show that determined partial automorphisms of finite trees are amalgamation bases (in the class of finite trees with partial automorphisms). To that end, we will show that the hypothesis of Lemma~\ref{lemma:determined_is_amalgamation_base} is satisfied, i.e.\ that the class of meet-trees with an automorphism has the amalgamation property; we divide the proof into several steps.
	\begin{definition}\leavevmode
		\label{definition:completion}
		\begin{itemize}[nosep]
			\item
			We say that a meet-tree is \emph{complete} if every chain has a least upper bound.
			\item
			Given a meet-tree $T$, its \emph{completion} $\hat T$ is the meet-tree consisting of cuts in $T$ (i.e.\ the downwards closed chains, cf.\ Remark~\ref{remark:cuts_in_trees}), ordered by inclusion, as in Remark~\ref{remark:ordering_cuts}.\xqed{\lozenge}
		\end{itemize}
	\end{definition}
	(Note that for the particular case of linear orders, the definition of completeness given above is slightly more stringent than the usual one: it implies that there is a maximal element.)

	\begin{remark}
		\label{remark:completion}
		If $(B,g),(C_1,f_1),(C_2,f_2)$ are trees with automorphisms, $B$ is downwards closed in $C_1$ and $C_2$, $B=C_1\cap C_2$ and $f_1\cap f_2=g$, then $\hat B=\hat C_1\cap \hat C_2$ and $\hat g=\hat f_1\cap \hat f_2$, and $\hat g,\hat f_1,\hat f_2$ are automorphisms of $\hat B,\hat C_1,\hat C_2$, respectively (where $\hat g(\hat b)=g[\hat b]$ etc.).\xqed{\lozenge}
	\end{remark}

	\begin{remark}
		\label{remark:bijection}
		If $B_1,B_2$ are sets, $f_1$ is a bijection on $B_1$, $f_2$ is a bijection on $B_2$, and $f_1$ and $f_2$ agree on $B=B_1\cap B_2$, then $f_1\cup f_2$ is a bijection on $B_1\cup B_2$: because they agree on $B$, $f_1\cup f_2$ is a well-defined function, and it is easy to check that $f_1^{-1}\cup f_2^{-1}$ is its inverse.\xqed{\lozenge}
	\end{remark}
	
	\begin{lemma}
		\label{lemma:down_closure}
		Suppose $(B,g),(C_1,f_1),(C_2,f_2)$ are meet-trees with automorphisms such that $B=C_1\cap C_2$ and $g=f_1\cap f_2$. Then there are $B',C_1',C_2',g',f_1',f_2'$ such that:
		\begin{itemize}
			\item
			$B'= C_1'\cap C_2'$ and is downwards closed in each of $C_1',C_2'$		,
			\item
			$g'$ is automorphism of $B'$ and $f_i'$ is an automorphism of $C_i'$ for $i=1,2$,
			\item
			$g'=f_1'\cap f_2'$,
			\item
			$(B,g)\subseteq (B',g')$ and $(C_i,f_i)\subseteq (C_i',f_i')$ for $i=1,2$.
		\end{itemize}
	\end{lemma}
	\begin{proof}
		Let $B_i'$ be the downwards closure of $B$ in $C_i$ for $i=1,2$. Put $B':=B_1'\cup B_2'$, $g':=f_1\restriction_{B_1'}\cup f_2\restriction_{B_2'}$, $C_i':=C_i\cup B'$ and $f_i':=f_i\cup g'$ for $i=1,2$.
		
		We put on each $B_i'$ the meet-tree structure inherited from $C_i$. We proceed to define the meet-tree structure on $B'$. Recall from Remark~\ref{remark:ordering_cuts} that we have a canonical ordering on the cuts in $B$. Now, take some $c_1,c_2$ from $B_1'\setminus B$ and $B_2'\setminus B$, respectively, and denote by $q_1,q_2$ their cuts in $B$, and take any $b_1,b_2\in B$ such that $b_i>c_i$. Then:
		\begin{itemize}
			\item
			if $q_1\leq q_2$, we declare $c_1<c_2$ and $c_1\meet c_2=c_1$,
			\item
			if $q_1>q_2$, we declare $c_1>c_2$ and $c_1\meet c_2=c_2$,
			\item
			if $q_1$ and $q_2$ are incomparable, we declare that $c_1$ and $c_2$ are incomparable and $c_1\meet c_2=b_1\meet b_2$.
		\end{itemize}
		The fact that this definition of $<$ defines a semilinear partial order is left as an exercise. Let us only show that the $\meet$ is correct and well-defined in the last case. We need to show that given $c\in B'$, we have that $c\leq c_1,c_2$ if and only if $c\leq b_1\meet b_2$ (this immediately implies that $c_1\meet c_2$ does not depend on the choice of $b_1$ and $b_2$). The fact that the left-hand side implies the right-hand side is trivial, since $b_i\geq c_i$. For the converse, suppose $c\leq b_1\meet b_2$. Note that $c_1$ and $c$ are comparable, so it is enough to show that $c\not> c_1$. But $c\leq b_2$, and we cannot have $c_1\leq b_2$ (because $c_1$ and $c_2\leq b_2$ are incomparable). The definition of the meet-tree structure on each $C_i'$ is straightforward.
		
		Now, we need to check that $g'$ is an automorphism of $B'$. It is clear that $f_i\restriction_{B_i'}$ is an automorphism of $B_i'$ for each $i$, and by Remark~\ref{remark:bijection}, $g'$ is a bijection, so we only need to check that it preserves the ordering and the meets. But this is easy to see as a consequence of the observation that $g'$ preserves $B,B_1'\setminus B, B_2'\setminus B$, and the cuts over $B$ of all elements of $B'$.
		
		Finally, again by Remark~\ref{remark:bijection}, each $f_i'$ is a bijection on $C_i'$, and the fact that it is an automorphism follows from Lemma~\ref{lemma:linear_union} (with $f=g'$ and $g=f_i$).
	\end{proof}
	
	\begin{lemma}
		\label{lemma:AP_in_trees_with_automorphisms}
		The class of meet-trees with automorphisms has the amalgamation property.
	\end{lemma}
	\begin{proof}
		Suppose $(B,g),(C_1,f_1),(C_2,f_2)$ are meet-trees with automorphisms, with fixed embeddings of $(B,g)$ into each $(C_i,f_i)$. We need to find a tree with automorphism $(D,h)$ such that $(C_i,f_i)$ embed in $(D,h)$ for $i=1,2$ in such a way that the two resulting embeddings of $(B,g)$ into $(D,h)$ coincide.
		
		We may assume for simplicity that $(B,g)\subseteq (C_i,f_i)$ for $i=1,2$, while $C_1\cap C_2=B$ and $f_1\cap f_2=g$. We will find a $(D,h)$ such that $(C_i,f_i)\subseteq (D,h)$ (which will immediately imply that the two embeddings of $(B,g)$ coincide).
		
		By Lemma~\ref{lemma:down_closure}, we may assume without loss of generality that $B$ is downwards closed in $C_1,C_2$. Then by Remark~\ref{remark:completion}, we may also assume that $B,C_1,C_2$ are complete. We put $D:=C_1\cup C_2$ and $h:=f_1\cup f_2$. We need to describe the meet-tree structure on $D$ and to show that $h$ is an automorphism of this structure. Now, for each $c\in C_1\cup C_2$, put $b_c:=\sup\{b\in B\mid b\leq c \}$ (this is well-defined by completeness of $B$).
		
		On $D$, we define the structure in the following way:
		\begin{itemize}
			\item
			on each $C_i$, the structure is simply the original structure,
			\item
			given $c_1\in C_1\setminus B$ and $c_2\in C_2\setminus B$, we declare $c_1$ and $c_2$ to be incomparable and put $c_1\meet c_2:=b_{c_1}\meet b_{c_2}$ (where the meet on the right-hand side is in the sense of $B$).
		\end{itemize}
		Note that this gives a meet-tree structure: indeed, transitivity of the order follows easily from the fact that $B$ is downwards closed in $C_1$ and $C_2$, and given any $d_0\in D$, the interval $(-\infty,d_0]$ in $D$ is contained in $C_1$ or $C_2$ (depending on whether $d_0\in C_1$ or $d_0\in C_2$), so semilinearity follows from semilinearity of $C_1$ and $C_2$. By the same token, given $c_1\in C_1$ and $c_2\in C_2$, the intersection $(-\infty,c_1]\cap (-\infty,c_2]$ is contained in $B$, so it is contained in $(-\infty,b_{c_1}]\cap (-\infty,b_{c_2}]=(-\infty,b_{c_1}\meet b_{c_2}]$.
		
		Furthermore, $h$ is an automorphism of $D$: by Remark~\ref{remark:bijection}, it is a bijection, and it clearly preserves meets and inequalities within each $C_i$. Now, given $c_1\in C_1\setminus C_2$ and $c_2\in C_2\setminus C_1$, the two are incomparable, and since $h(c_1)=f_1(c_1)\in C_1\setminus C_2$ and $h(c_2)=f_2(c_2)\in C_2\setminus C_1$ (because $C_1,C_2$ are clearly $h$-invariant), $h(c_1)$ and $h(c_2)$ are incomparable. It is also not hard to see that $b_{f_i(c_i)}=f_i(b_{c_i})$, which implies that $h$ preserves meets.
		
		Finally, clearly $D\supseteq C_1,C_2$ and $h\supseteq f_1,f_2$, so $(D,h)$ is as desired.
	\end{proof}
	
	One could ask whether the analogue of Lemma~\ref{lemma:AP_in_trees_with_automorphisms} for trees of bounded arity is true. Unfortunately, this is not the case, which the following proposition shows.
	\begin{proposition}
		\label{proposition:no_AP_bdd}
		For every integer $k>1$, the class of meet-trees of arity at most $k$ with an automorphism does not have the amalgamation property.
	\end{proposition}
	\begin{proof}
		Fix $k$ and let $B=\{b\}$, $C_1=B\cup \{c_1^1,c_1^2,\ldots,c_1^k\}$, $C_2=B\cup \{c_2\}$ be meet-trees such that $c_2>b$ and for all $i\neq j$ we have $c_1^i,c_1^j>c_1^i\meet c_1^j=b$. Define $f$ on $B$ as $f(b)=b$, $f_1\supseteq f$ on $C_1$ as $f_1(c_1^i)=c_1 ^{i+1}$, for $i<k$, $f_1(c_1^k)=c_1^1$, and finally define $f_2\supseteq f$ on $C_2$ as the identity map.
		
		We  claim that $(C_1,f_1)$ and $(C_2,f_2)$ do not amalgamate over $(B,f)$. Indeed, suppose we do have an amalgam $(D,g)$, with $D$ being a tree of arity at most $k$. It follows that for some $i$ we have $c_1^i\meet c_2>b$. We may assume without loss of generality that $i=1$. But then by Fact~\ref{fact:lemmaxyz}, since $b=c_1^1\meet c_1^2$, it follows that $c_1^2\meet c_2=b$. On the other hand, since $g$ is an automorphism of $D$, we have $b=g(b)<g(c_1^1\meet c_2)=g(c_1^1)\meet g(c_2)=c_1^2\meet c_2$, a contradiction.
	\end{proof}
	In the opposite direction to Proposition~\ref{proposition:no_AP_bdd}, the following Remark describes a possible class of amalgamation classes for the case of bounded arity.
	\begin{remark}
		\label{remark:sufficient_for_AP_bdd}
		We suspect that the following is true (but, as we do not use it, we will did not check it very carefully): if $k>1$ and $(B,g)$ is a meet-tree of arity at most $k$ with an automorphism, such that every $g$-periodic $b\in B$ has maximal rank (i.e.\ there are $b_1,\ldots, b_k$ which are pairwise incomparable, such that $b_i\meet b_j=b$ for all $i\neq j$), then $(B,g)$ is an amalgamation base (in the class of all meet-trees of arity at most $k$ with an automorphism).\xqed{\lozenge}
	\end{remark}
	
	The following corollary summarises the results of this section. In the rest of the paper, we will show (using Lemma~\ref{lemma:sufficient_for_determined}) that determined finite partial automorphisms of $\bT$ are cofinal, which will imply that $\bT$ has a generic automorphism.
	\begin{corollary}
		\label{corollary:determined_is_AB_trees}
		If $p$ is a finite partial automorphism of $\bT$ and $B\subseteq \bT$ is generated by $\dom(p)\cup \range(p)$, then $(B,p)$ is an amalgamation base in the class of finite trees with partial automorphisms if and only if $p$ is determined.
	\end{corollary}
	\begin{proof}
		Note that if the age of a structure consists of finite meet-trees, then it is a meet-tree itself, and conversely, the age of any meet-tree consists of finite meet-trees. The corollary follows immediately by Lemma~\ref{lemma:AP_in_trees_with_automorphisms}, Lemma~\ref{lemma:determined_is_amalgamation_base} and Proposition~\ref{proposition:amalgamation_implies_determined}.
	\end{proof}
	
	\begin{remark}
		\label{remark:determined_insufficient_bdd}
		Note that the construction from Proposition~\ref{proposition:no_AP_bdd} actually shows that the analogue of Corollary~\ref{corollary:determined_is_AB_trees} for $\bT_k$ fails (when $k>1$): an automorphism of $\bT_k$ which is just a single fixed point is trivially determined, and the example shows that it does not yield an amalgamation base.\xqed{\lozenge}
	\end{remark}
	
	\begin{remark}
		\label{remark:AP_linear_aut}
		One can show that the class of linear orders with an automorphism does have AP (essentially, arguing as in the proof of Lemma~\ref{lemma:down_closure}), so the analogue of Corollary~\ref{corollary:determined_is_AB_trees} for $\bT_1$ holds. This can be used (along with a variant of Theorem~\ref{theorem:main}) to recover the fact that $({\mathbf Q},<)$ has a generic automorphism, see Remark~\ref{remark:works_for_linear}.\xqed{\lozenge}
	\end{remark}

	\section{Orbits in meet-trees}
	\label{section:Orbits in meet-trees}
	Having Corollary~\ref{corollary:determined_is_AB_trees}, to show CAP for the class of finite automorphisms of meet-trees, it is enough to show that determined partial automorphisms are cofinal. Before we can do that, we need to understand the orbits of partial automorphisms in trees and their extensions.
	\begin{definition}
		\label{definition:partial_orbit}
		A \emph{(finite) partial orbit} (in a meet-tree) is a finite sequence $\eta=(\eta_0,\eta_1,\ldots,\eta_n)$ of elements of the tree such that the map $\eta_i\mapsto \eta_{i+1}$ for $i=0,\ldots, n-1$ is a partial automorphism.
		
		Given a finite partial automorphism $p$, an \emph{orbit of $p$} is a partial orbit $\eta=(\eta_0,\ldots, \eta_n)$ such that $p(\eta_i)=\eta_{i+1}$ for all $i<n$ and either $\eta_0=\eta_k$ for some $k\leq n$, or $\eta_0\notin \range(p)$ and $\eta_n\notin \dom(p)$.
		
		Infinite partial orbits and orbits of infinite partial automorphisms are defined analogously.\xqed{\lozenge}
	\end{definition}
	Until the end of the paper, we will use the convention that lower-case Greek letters represent (usually finite) partial orbits in meet-trees, and each such orbit $\eta$ is enumerated as $\eta_0,\ldots,\eta_n$, as in Definition~\ref{definition:partial_orbit}.
	
	\begin{remark}
		Since every finite tree can be embedded in $\bT$ (which is ultrahomogeneous), every finite partial meet-tree automorphism can be extended to a total meet-tree automorphism (possibly after enlarging the tree). In particular, every finite partial orbit is contained in the orbit of some total automorphism.
		
		(Likewise, every finite automorphism and finite partial orbit in a meet-tree of arity at most $k$ can be extended to an automorphism or a full orbit of a meet-tree of arity at most $k$.)\xqed{\lozenge}
	\end{remark}
	
	\begin{definition}
		\label{definition:orbit_types}
		Let $\eta=(\eta_0,\eta_1,\ldots,\eta_n)$ be a partial orbit, while $k$ is a positive integer.
		
		Then:
		\begin{itemize}
			\item We say that $\eta$ is a \emph{$k$-cycle} (or \emph{$k$-cyclic}) if $k$ is minimal such that for some $\eta_k=\eta_0$.
			
			\item We say that $\eta$ is an \emph{ascending $k$-spiral} if $k$ is minimal such that 
			$\eta_k>\eta_0$.
			Likewise, we say that it is a \emph{descending $k$-spiral} if $k$ is minimal such that 
			$\eta_k<\eta_0$.
			
			\item We say that $\eta$ is an
			\emph{ascending $k$-comb} if it is not a spiral and $k$ is minimal such that $\eta_{2k}\meet \eta_k>\eta_k\meet \eta_0$. Likewise, we say that it is a \emph{descending $k$-comb} if it is not a spiral and $k$ is minimal such that $\eta_{2k}\meet \eta_k<\eta_k\meet \eta_0$.
			
			\item Otherwise, if $\eta$ is not a cycle, spiral, nor a comb, we say that $\eta$ is a \emph{quasi-cycle} (or \emph{quasi-cyclic}).
		\end{itemize}
		If $\eta$ is a $k$-cycle, we say that $k$ is the \emph{period} of $\eta$; if $\eta$ is a  $k$-spiral or a $k$-comb, we say that $k$ is its \emph{spiral length}.
		
		We define the \emph{length} of $\eta$ as the size of $\{\eta_0,\ldots,\eta_n\}$.
		
		(See Figures~\ref{figure:spiral_comb} and~\ref{figure:quasi_cycle} for some examples of orbits of various types.)\xqed{\lozenge}
	\end{definition}
	
	\begin{remark}
		\label{remark:quasi_cycle_meets_cycle}
		Note that if $\eta=(\eta_0,\ldots,\eta_n)$ is a cycle or a quasi-cycle, then for all $k\leq n/2$, we have $\eta_0\meet \eta_k=\eta_k\meet \eta_{2k}$.\xqed{\lozenge}
	\end{remark}

	\begin{remark}
		\label{remark:trivial_preservation}
		If $\eta\subseteq \bar \eta$ are finite partial orbits, then if $\eta$ is a $k$-cycle, an ascending/descending $k$-spiral or $k$-comb, then so is $\bar \eta$ (respectively). It follows that a suborbit of a quasi-cycle is a quasi-cycle. On the other hand, if $\eta$ is a quasi-cycle, then it does not say much about the kind of orbit $\bar\eta$ can be (see Remark~\ref{remark:can_do_anything}).\xqed{\lozenge}
	\end{remark}
	
	\begin{remark}
		\label{remark:time_reversal}
		If $\eta=(\eta_0,\ldots,\eta_n)$ is a partial orbit, then so is $\eta^{-1}=(\eta_n,\eta_{n-1},\ldots,\eta_0)$. If we think of $\eta$ as a sequence of steps in a process, then $\eta^{-1}$ can be thought of as ``time-reversal'' of this process.\xqed{\lozenge}
	\end{remark}
	
	\begin{figure}
		\center
		\begin{tikzpicture}[main_node/.style={circle,fill=gray!40,draw,minimum size=1em,inner sep=3pt]},other_node/.style={circle,pattern=north west lines, pattern color=gray!40,draw,minimum size=1em,inner sep=3pt]}]
			\node[main_node] (root) at (3,0) {$\,$};
			
			\node[other_node] (zeta0) at (2,1) {$\zeta_0$};
			\node[other_node] (zeta1) at (4,1) {$\zeta_1$};
			
			\foreach \x in {0,1,2}
			\foreach \y in {0,1,2,3}
			{
				\pgfmathtruncatemacro{\label}{4*\x+\y}
				\pgfmathtruncatemacro{\px}{\x-1}
				\ifthenelse{\x < 2 \OR \y <2}{
					\node[main_node] (eta\x\y) at (2*\y,3+2*\x) {$\eta_{\label}$};
					\ifthenelse{\x>0}{\draw (eta\x\y)-- (eta\px\y);}{}
				}{}
				\ifthenelse{\x < 2 \OR \y = 0}{
					\node[other_node] (mu\x\y) at (2*\y+1,3+2*\x+1) {$\mu_{\label}$};
					\draw (eta\x\y)-- (mu\x\y);
				}{}
				
			}
			\foreach \x in {0,1}
			{
				\pgfmathtruncatemacro{\xp}{\x+2}
				\draw (eta0\x)--(zeta\x)--(eta0\xp);
				\draw (root)--(zeta\x);
			}
		\end{tikzpicture}
		\caption{In the above figure, the $\zeta$-orbit is a $2$-cycle, the $\eta$-orbit is an ascending $4$-spiral, while the $\mu$-orbit is an ascending $4$-comb. The root is a fixed point.}
		\label{figure:spiral_comb}
	\end{figure}
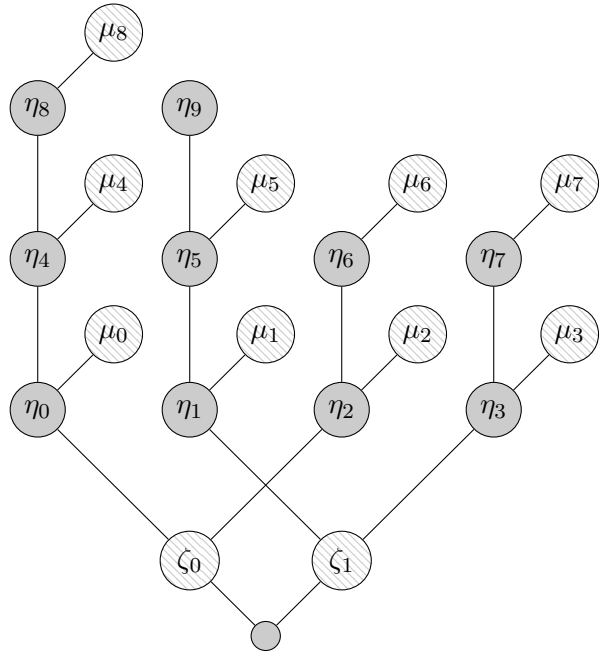

	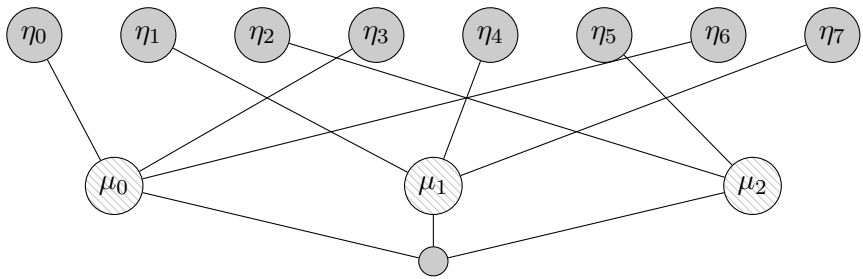
\begin{figure}
		\center
		\begin{tikzpicture}[main_node/.style={circle,fill=gray!40,draw,minimum size=1em,inner sep=3pt]},other_node/.style={circle,pattern=north west lines, pattern color=gray!40,draw,minimum size=1em,inner sep=3pt]}]
			\node[main_node] (root) at (7*1.5/2,0) {$\,$};

			\foreach \x in {0,...,7}
			{
				\node[main_node] (eta\x) at (1.5*\x,3) {$\eta_\x$};
			}
			\foreach \x in {0,...,2}
			{	
				\pgfmathsetmacro{\part}{7*1.5/10}
				\node[other_node] (mu\x) at (\part*4*\x+\part,1) {$\mu_\x$};
				\draw (root) --(mu\x);
				\foreach \y in {0,1,2}{
					\pgfmathtruncatemacro{\yp}{\x+3*\y}
					\ifthenelse{\yp<8}{
						\draw (mu\x)--(eta\yp);
					}{}
				}
				
			}
		\end{tikzpicture}
		\caption{In the above figure, the $\eta$-orbit is a quasi-cycle of pseudo-period $3$ (see Definition~\ref{definition:pseudo-period} below), the $\mu$-orbit is a $3$-cycle, while the root is a fixed point.}
		\label{figure:quasi_cycle}
	\end{figure}

	The following proposition describes the spiral behaviour of orbits.
	\begin{proposition}
		\label{proposition:spirals}
		Suppose that $\eta=(\eta_0,\ldots,\eta_n)$ is $k$-spiral. Then for $i,j,l\in [0,n]$:
		\begin{enumerate}[label=(\arabic{*})]
			\item
			$\eta_i$ is comparable to $\eta_j$ if and only if $i\equiv j\pmod k$,
			\item
			if $i\equiv j\not\equiv l\pmod k$, then $\eta_i\meet \eta_l=\eta_j\meet \eta_l$,
			\item
			if $2l\leq n$ and $\eta_0\meet \eta_l\neq \eta_l\meet \eta_{2l}$, then $k$ divides $l$ (so $\eta_0<\eta_l<\eta_{2l}$ or $\eta_0>\eta_l>\eta_{2l}$).
		\end{enumerate}
	\end{proposition}
	\begin{proof}
		The case of $k=1$ is straightforward, so we may assume that $k>1$. We may assume without loss of generality that $\eta$ is ascending (replacing it with $\eta^{-1}$ if necessary), and we may also assume that $n\geq jk$ (extending $\eta$ if necessary).
		
		(1): It is clear (by straightforward induction) that if $i\equiv j\pmod k$, then $\eta_i$ and $\eta_j$ are comparable. For the converse, we may assume without loss of generality that $i<j$, and in fact that $i=0$. Since $\eta$ is an ascending $k$-spiral, $\eta_0<\eta_{jk}$. This clearly implies that $\eta_j\not\leq \eta_0$ (otherwise, we would have $\eta_{jk}\leq \eta_0$), so by comparability, $\eta_j>\eta_0$. Let $m$ be maximal such that $km\leq j$. Since $j>0$ and $k$ is minimal such that $\eta_k>\eta_0$, we have $m>0$. Now, it is easy to see that we have $\eta_{km},\eta_j<\eta_{jk}$, so $\eta_{km},\eta_j$ are comparable. It follows that $\eta_0,\eta_{j-km}$ are comparable. But $0\leq j-km<k$, so by the choice of $k$, it follows that $j-km=0$, so $k$ divides $j$.
		
		(2): By (1), $\eta_i$ and $\eta_j$ are comparable, so without loss of generality we can assume $\eta_i\leq \eta_j$. It follows that $\eta_i\meet \eta_j=\eta_i$. Also by (1), $\eta_i,\eta_l$ are not comparable, so $\eta_i\meet\eta_l<\eta_i$. Together, we have that $\eta_i\meet \eta_j>\eta_i\meet \eta_l$, so by Fact~\ref{fact:lemmaxyz}, $\eta_i\meet \eta_l=\eta_j\meet \eta_l$.
		
		(3): Write $N$ for the least common multiple of $l$ and $k$; we may assume without loss of generality that $N\leq n$ (extending $\eta$ if necessary).
		
		Consider for simplicity the case when $\eta_0\meet \eta_l>\eta_l\meet \eta_{2l}$ (the other case is similar). If $l=N$, we are done. Otherwise, by straightforward induction, $\eta_0\meet \eta_l>\eta_l\meet \eta_{N}$, so in particular $\eta_0\meet \eta_l\neq \eta_N\meet \eta_l$. But $k$ divides both $0$ and $N$, so by (2) we conclude that $k$ divides $l$. (This also contradicts the hypothesis that $l\neq N$.)
	\end{proof}
	
	\begin{proposition}
		\label{proposition:needed_comb}
		Let $\eta=(\eta_0,\ldots,\eta_n)$ be a partial orbit. Take any positive $k\leq n/2$ and integers $i_1\equiv i_2\equiv j_1\equiv j_2\pmod k$ such that $i_1<i_2$ and $j_1<j_2$.
		
		If $\eta_0\meet \eta_k<\eta_k\meet \eta_{2k}$ (in particular, if $\eta$ is an ascending $k$-spiral or $k$-comb), then:
		\begin{enumerate}
			\item
			\begin{enumerate}
				\item
				\label{it:needed_comb_1a}
				$\eta_{i_1}\meet \eta_{i_2}=\eta_{i_1}\meet \eta_{i_1+k}$;
				\item
				\label{it:needed_comb_1b}
				$\otp(\eta_{i_1}\meet \eta_{i_2},\eta_{j_1}\meet \eta_{j_2})=\otp(i_1,j_1)$.
			\end{enumerate}
		\end{enumerate}
		
		Likewise, if $\eta_0\meet \eta_k>\eta_k\meet \eta_{2k}$ (in particular, if $\eta$ is a descending $k$-spiral or $k$-comb), then:
		\begin{enumerate}[resume]
			\item
			\begin{enumerate}
				\item
				\label{it:needed_comb_2a}
				$\eta_{i_1}\meet \eta_{i_2}=\eta_{i_2-k}\meet \eta_{i_2}$;
				\item
				\label{it:needed_comb_2b}
				$\otp(\eta_{i_1}\meet \eta_{i_2},\eta_{j_1}\meet \eta_{j_2})=\otp(j_2,i_2)$.
			\end{enumerate}
		\end{enumerate}
	\end{proposition}
	\begin{proof}
		We will consider the ascending case (the descending case is completely analogous, and in fact it follows by considering the ``time reversal''). \eqref{it:needed_comb_1a} follows by induction with respect to $i_2-i_1$; if $i_2-i_1=k$, then the conclusion is trivial. Otherwise, if $i_2>i_1+k$, by the induction hypothesis, we have that $\eta_{i_1+k}\meet \eta_{i_2}=\eta_{i_1+k}\meet \eta_{i_1+2k}$. Since $\eta_0\meet \eta_k<\eta_k\meet \eta_{2k}$, also $\eta_{i_1}\meet \eta_{i_1+k}<\eta_{i_1+k}\meet \eta_{i_1+2k}$, whence $\eta_{i_1}\meet \eta_{i_1+k}<\eta_{i_1+k}\meet \eta_{i_2}$, and the conclusion follows by Fact~\ref{fact:lemmaxyz}.
		
		For \eqref{it:needed_comb_1b}, by \eqref{it:needed_comb_1a}, we may assume without loss of generality that $i_2=i_1+k$ and $j_2=j_1+k$. We may also assume without loss of generality that $i_1<j_1$. Straightforward induction shows that $\eta_{i_1}\meet \eta_{i_1+k}<\eta_{j_1}\meet \eta_{j_1+k}$, which finishes the proof.
	\end{proof}

	\begin{remark}
		\label{remark:suff_p_aut}
		If $p$ is a partial automorphism of a first-order structure $C$ and $a,b\in C$, then $p\cup\{(a,b)\}$ is a partial automorphism of $C$ if and only if $\qftp(b/\range(p))=p(\qftp(a/\dom(p))$.\xqed{\lozenge}
	\end{remark}

	\begin{definition}
		\label{definition:pseudo-period}
		Given a finite partial orbit $\eta=(\eta_0,\ldots,\eta_n)$ ($n>0$), its \emph{pseudo-period} is the smallest $u>0$ such that $\eta_0\meet \eta_u=\max_{0<i\leq n} \eta_0\meet \eta_i$.\xqed{\lozenge}
	\end{definition}
	
	\begin{remark}
		Note that if $\eta$ is a cycle, then its period is also its pseudo-period.\xqed{\lozenge}
	\end{remark}
	
	\begin{proposition}
		\label{proposition:qc_time_reversal} The pseudo-period is invariant under time reversal.
		
		More precisely, if $\eta=(\eta_0,\ldots,\eta_n)$ is a partial orbit of pseudo-period $u$, then $u$ is also the pseudo-period of $\eta^{-1}=(\eta_n,\eta_{n-1},\ldots,\eta_1,\eta_0)$, i.e.\ it is the smallest $m_0>0$ such that $\eta_{n-m_0}\meet \eta_n=\max_{0<m\leq n} \eta_n\meet \eta_{n-m}$.
	\end{proposition}
	\begin{proof}
		Let $u'$ be the pseudo-period of $\eta^{-1}$. By symmetry, it is enough to show that $u'\geq u$. This is equivalent to saying that for all positive $m<u$, we have $\eta_n\meet \eta_{n-m}<\eta_n\meet \eta_{n-u}$. Note that for such $m$ we have $0<u-m<u$, so $\eta_0\meet \eta_{u-m}<\eta_0\meet \eta_u$. Now, let $f$ be an automorphism of a meet-tree such that for $f(\eta_i)=\eta_{i+1}$ for $i=0,\ldots, n-1$. By applying $f^{n-u}$, we obtain that $\eta_{n-u}\meet \eta_{n-m}<\eta_{n-u}\meet \eta_n$. By Fact~\ref{fact:lemmaxyz}, it follows that $\eta_{n-u}\meet \eta_{n-m}=\eta_n\meet \eta_{n-m}$ and so we are done.
	\end{proof}
	
	\begin{proposition}
		\label{proposition:baby_cycle}
		Suppose $\eta=(\eta_0,\ldots,\eta_n)$ is a quasi-cycle of pseudo-period $u$. Then either $n\geq 2u$ or $\eta$ can be extended to $\bar\eta=(\eta_0,\ldots,\eta_n,\eta_{n+1},\ldots,\eta_{2u})$, a partial orbit such that $\eta_{2u}=\eta_0$ (so $\bar\eta$ is a $2u$-cycle).
		
		As a consequence, $\eta$ is an orbit of a finite partial automorphism $p$ such that $\eta_0\meet \eta_u$ is a fixed point of $p^u$ (in particular, $p^u(\eta_0\meet \eta_u)$ is defined).
	\end{proposition}
	\begin{proof}
		Note first that the ``as a consequence'' part follows easily from the statement by Remark~\ref{remark:quasi_cycle_meets_cycle}.
		
		Let $f$ be a meet-tree automorphism, one of whose orbits includes $\eta$. For each $i>n$, put $\eta_i:=f^i(\eta_0)$. We may assume without loss of generality that $n+1\leq 2u$ (otherwise, the conclusion is trivial).
		
		\begin{claim*}
			$(\eta_0,\ldots,\eta_{2u-1})$ is a quasi-cycle of pseudo-period $u$.
		\end{claim*}
		\begin{clmproof}
			First, to show that the pseudo-period is the same, notice that for $u<i<2u$ we have (by Proposition~\ref{proposition:qc_time_reversal} applied to the partial orbit $(\eta_{i-u},\ldots, \eta_i)$) $\eta_{i-u}\meet \eta_i>\eta_u\meet \eta_i$, so (by Fact~\ref{fact:lemmaxyz} and another application of Proposition~\ref{proposition:qc_time_reversal}) $\eta_u\meet \eta_i=\eta_u\meet \eta_{i-u}<\eta_0\meet \eta_u$, and hence $\eta_0\meet \eta_i=\eta_u\meet \eta_i<\eta_0\meet \eta_u$.
			
			We will prove by induction that $(\eta_0,\ldots,\eta_m)$ is a quasi-cycle, where $m<2u$. If $m\leq n$, then there is nothing to prove. Suppose now that $n\leq m<2u-1$ and we know that $(\eta_0,\ldots,\eta_m)$ is a quasi-cycle, and we show that so is $(\eta_0,\ldots,\eta_{m+1})$.
			
			First, we show that $\eta_{m+1}\not\geq \eta_0$ --- the argument for $\eta_{m+1}\not\leq \eta_{0}$ is symmetric, in light of Proposition~\ref{proposition:qc_time_reversal}, and the two together show that we have neither a cycle nor a spiral. Note that since $m+1<2u$, we have $m+1-u<u$. It follows immediately that $\eta_0\meet \eta_u>\eta_0\meet \eta_{m+1-u}$, and by Proposition~\ref{proposition:qc_time_reversal}, it also follows that $\eta_{m+1}\meet \eta_{m+1-u}>\eta_{m+1}\meet \eta_u$. This clearly implies that either $\eta_0\meet \eta_{m+1-u}< \eta_{m+1}\meet \eta_{m+1-u}$ or $\eta_{m+1}\meet \eta_u< \eta_0\meet \eta_u$ (having in mind that the two sides of each inequality are comparable by semilinearity).
			
			Suppose first that $\eta_0\meet \eta_{m+1-u}< \eta_{m+1}\meet \eta_{m+1-u}$. By Fact~\ref{fact:lemmaxyz}, $\eta_0\meet \eta_{m+1-u}=\eta_0\meet \eta_{m+1}$. But note that since $\eta$ is a quasi-cycle (and $m+1-u\leq n$), $\eta_0\meet \eta_{m+1-u}<\eta_0$, so it follows that $\eta_0\not\leq \eta_{m+1}$. Otherwise, if $\eta_{m+1}\meet \eta_u< \eta_0\meet \eta_u$, then analogously $\eta_{m+1}\meet \eta_u=\eta_0\meet \eta_{m+1}=\eta_0\meet \eta_u\meet \eta_{m+1}$, and since $\eta_0>\eta_0\meet \eta_u$ (because $\eta$ is a quasi-cycle), we have $\eta_0>\eta_0\meet \eta_u\meet \eta_{m+1}=\eta_0\meet \eta_{m+1}$, so $\eta_{m+1}\not\geq \eta_0$.
			
			Now, if, in addition, we have that $m$ is odd, i.e.\ $m+1=2c$ for some $c<u$, then we need to show that $(*)\,\eta_0\meet \eta_c=\eta_c\meet \eta_{2c}$. Since $c<u\leq m=2c-1$, we have $2c>u$, so $c>u-c$, whence $m=2c-1\geq 2(u-c)$. Thus, by the induction hypothesis, $\eta_0\meet\eta_{u-c}=\eta_{u-c}\meet \eta_{2u-2c}$. On the other hand, $0<u-c<u$, so $\eta_0\meet \eta_{u-c}<\eta_0\meet \eta_u$, whence $\eta_0\meet \eta_{u-c}=\eta_u\meet \eta_{u-c}$, so $\eta_u\meet \eta_{u-c}=\eta_{u-c}\meet \eta_{2u-2c}$. By applying $f^{2c-u}$, we obtain $(**)\,\eta_{2c}\meet \eta_c=\eta_c\meet \eta_u$. But since $c<u$, we have also $\eta_0\meet \eta_c<\eta_0\meet \eta_u$, so by another application of Fact~\ref{fact:lemmaxyz}, $(***)\,\eta_0\meet \eta_c=\eta_u\meet \eta_c$. $(**)$ and $(***)$ together imply $(*)$, which shows that $(\eta_0,\ldots,\eta_{2u-1})$ is a quasi-cycle, completing the proof of the claim.
		\end{clmproof}
		
		We will show that $\bar\eta=(\eta_0,\ldots,\eta_{2u-1},\eta_0)$ is a partial orbit (which is clearly sufficient).
		
		In light of the claim, we may assume that $n=2u-1$. By Remark~\ref{remark:suff_p_aut}, to complete the proof, it is enough to show that $\qftp(\eta_0/\eta_1,\ldots,\eta_{2u-1})=\qftp(\eta_{2u}/\eta_1,\ldots,\eta_{2u-1})$. Since $(\eta_0,\ldots,\eta_n)$ is a quasi-cycle, it follows that $\eta_0,\eta_u>\eta_0\meet\eta_u$, and hence also $\eta_u,\eta_{2u}>\eta_{u}\meet\eta_{2u}$. We will show that for every $1\leq i,j\leq n$, we have that $b:=\eta_i\meet \eta_j<\eta_0\meet \eta_u$ if and only if $b<\eta_u$, if and only if $b<\eta_u\meet \eta_{2u}$. Since $\eta$ is a quasi-cycle, it easily follows that we cannot have $b>\eta_u$, so by Fact~\ref{fact:typesintrees} (with $a=\eta_u$ and $b'$ equal to $\eta_0\meet \eta_u$ or $\eta_u\meet \eta_{2u}$, using Proposition~\ref{proposition:qc_time_reversal}), it will follow that the types are equal as required.
		
		It is clear that if $b<\eta_0\meet \eta_u$ or $b<\eta_u\meet \eta_{2u}$, then $b<\eta_u$, so suppose that $b<\eta_u$, and let us show that $b<\eta_u\meet \eta_{2u}$ and $b<\eta_0\meet \eta_{u}$. Note that if $b=\eta_i\meet \eta_j$, then by Fact~\ref{fact:lemmaxyz}, it follows that either $b=\eta_u\meet \eta_i$ or $b=\eta_u\meet \eta_j$. Without loss of generality we may assume the former.
		
		Now, since $b\neq \eta_u$, we have $i\neq u$. Let us consider the case when $i>u$ (the other case is analogous). Write $k$ for $i-u$, so that $i=u+k$. Clearly, $\eta_0\meet \eta_u>\eta_0\meet \eta_k$, so $\eta_u\meet \eta_{2u}>\eta_{u+k}\meet \eta_u=\eta_i\meet \eta_u$. Furthermore, $0<k<u$, so $\eta_0\meet \eta_u>\eta_0\meet \eta_{u-k}$, so $\eta_0\meet \eta_{u-k}=\eta_u\meet \eta_{u-k}$. whence $\eta_k\meet \eta_u=\eta_{u+k}\meet \eta_u$. But $\eta_0\meet \eta_u>\eta_0\meet \eta_k$, so $\eta_0\meet \eta_k=\eta_u\meet \eta_k=\eta_{u+k}\meet \eta_u=\eta_i\meet \eta_u$ (since $u+k=i$) and we are done.
	\end{proof}
	
	\begin{proposition}
		\label{proposition:needed_qc}
		Let $\eta=(\eta_0,\ldots,\eta_n)$ be a quasi-cycle with pseudo-period $u$.
		
		Then if $i,j,k\in \{0,\ldots,n\}$ satisfy $i\equiv j\pmod u$ and $k\neq i,j$, then $\eta_i\meet \eta_k=\eta_j\meet \eta_k$.
	\end{proposition}
	\begin{proof}
		By Proposition~\ref{proposition:qc_time_reversal}, we may assume that $j<k$. Indeed, otherwise, if $j>k$, then we can simply consider $\eta^{-1}$ --- this preserves $u$, but reverses the order of $\eta_j$ and $\eta_k$ in the orbit. Further, we may also assume without loss of generality that $j<i$ (otherwise, if $i<j<k$, we can just swap $i$ and $j$; the case of $i=j$ is trivial). Moreover, it is enough to consider the case when $j=0$ (truncating $\eta$ if necessary). This leaves us with some $i,k>0$ such that $i\neq k$ and $u$ divides $i$, and we need to show that $\eta_0\meet \eta_k=\eta_i\meet \eta_k$.
		
		By Proposition~\ref{proposition:baby_cycle}, we have a finite partial automorphism $p$ such that $\eta$ is an orbit of $p$ and $p^u(\eta_0\meet \eta_u)=\eta_0\meet \eta_u$. Let $f\supseteq p$ be a total meet-tree automorphism.
		\begin{claim*}
			If $0<m\leq n$, then $\eta_0\meet \eta_m=\eta_0\meet \eta_u$ if and only if $u$ divides $m$.
		\end{claim*}
		\begin{clmproof}
			Suppose first that $u$ divides $m$, so $m=lu$. The proof is by induction with respect to $l$. Since $\eta$ is a quasi-cycle, we have that $\eta_{lu-u}\meet \eta_{lu}=\eta_{lu-2u}\meet \eta_{lu-u}=\ldots=\eta_0\meet \eta_u$, so in particular, $\eta_0\meet \eta_u=\eta_{lu-u}\meet \eta_{lu}$. By the induction hypothesis, $\eta_0\meet\eta_u=\eta_0\meet \eta_{lu-u}$, so $\eta_0\meet\eta_{lu-u}=\eta_{lu-u}\meet \eta_{lu}$. By Fact~\ref{fact:lemmaxyz}, it follows that $\eta_0\meet \eta_{lu}\geq \eta_0\meet \eta_{lu-u}=\eta_0\meet \eta_u$. Since by the definition of $u$, $\eta_0\meet \eta_u\geq \eta_0\meet \eta_{lu}$, this completes the proof.
			
			Now, suppose $\eta_0\meet \eta_m=\eta_0\meet \eta_u$. Let $l$ be maximal such that $lu\leq m$. Then, by the preceding paragraph, $\eta_0\meet \eta_{lu}=\eta_0\meet \eta_{u}$, so by Fact~\ref{fact:lemmaxyz}, $\eta_{lu}\meet \eta_m\geq \eta_0\meet \eta_u$. Now, $\eta_0\meet \eta_u$ is a fixed point of $f^u$, and hence also of $p^{-lu}$. By applying $f^{-lu}$ to the last inequality, we obtain $f^{-lu}(\eta_{lu}\meet \eta_m)=\eta_0\meet \eta_{m-lu}\geq \eta_0\meet \eta_u$. But $m-lu<u$, so by minimality of $u$, we have that $m-lu=0$, so $m=lu$.
		\end{clmproof}
		Now, by Claim, we have $\eta_0\meet \eta_i=\eta_0\meet \eta_u\geq\eta_0\meet \eta_k$. If the inequality is strict, by Fact~\ref{fact:lemmaxyz}, $\eta_0\meet \eta_k=\eta_i\meet \eta_k$ and we are done. Otherwise, by Claim, $u$ divides $k$ and $\eta_0\meet \eta_k=\eta_0\meet \eta_u$, so we need to show that $\eta_i\meet \eta_k=\eta_0\meet \eta_u$.
		
		We have that $u$ divides $k-i$ so by Claim, we have $\eta_0\meet \eta_{\lvert k-i\rvert}=\eta_0\meet \eta_u$. Since $\eta_0\meet \eta_u$ is a fixed point of $p^u$ and $u$ divides $i,k$ (and hence also $\min(i,k)$), $\eta_0\meet \eta_u$ is also a fixed point of $p^{\min(i,k)}$, so we have that
		\[
		\eta_i\meet \eta_k=\eta_{\min(i,k)}\meet \eta_{\min(i,k)+\lvert k-i\rvert}=p^{\min(i,k)}(\eta_0\meet \eta_{\lvert k-i\rvert})=p^{\min(i,k)}(\eta_0\meet \eta_u)=\eta_0\meet \eta_u,
		\]
		so we are done.
	\end{proof}
	
	The next two corollaries are related to Question~\ref{question:not_qc_ext}, and they are not necessary to prove the main result of the paper, Theorem~\ref{theorem:main}.
	
	\begin{corollary}
		\label{corollary:qc_extension}
		If $\eta=(\eta_0,\ldots,\eta_n)$ is a quasi-cycle with pseudo-period $u$, and $N>n$ is minimal such that $u$ divides $N$, then every extension $\bar\eta=(\eta_0,\ldots,\eta_{N-1})$ is also a quasi-cycle with pseudo-period $u$.
	\end{corollary}
	\begin{proof}
		The case of $N=2u$ is the claim in the proof of Proposition~\ref{proposition:baby_cycle}, so we may assume that $n\geq 2u$.
		
		First, we will show that the pseudo-period is indeed $u$. Take any $m\in(n,N)$. Then $u$ does not divide $m$ and $2u\leq n<m< n+u$, so $(\eta_u,\ldots,\eta_m)$ is a quasi-cycle with pseudo-period $u$; take $m'\equiv m \pmod u$ such that $u<m'<2u$, so that Proposition~\ref{proposition:needed_qc} implies that $\eta_u\meet \eta_m=\eta_u\meet \eta_{m'}$, so $\eta_u\meet \eta_m=\eta_u\meet \eta_{m'}<\eta_u\meet \eta_{2u}$. Then $\eta_u\meet \eta_{2u}=\eta_0\meet\eta_u$, so $\eta_u\meet\eta_m<\eta_0\meet\eta_u$, and so, by Fact~\ref{fact:lemmaxyz}, $\eta_0\meet \eta_m=\eta_u\meet\eta_m<\eta_0\meet\eta_u$.
		
		Now, we will show that $(\eta_0,\ldots,\eta_m)$ is a quasi-cycle by induction with respect to $m\geq n$. The case of $m= n$ is clear. Suppose now that $n<m<N-1$ and $(\eta_0,\ldots,\eta_m)$ is a quasi-cycle (with pseudo-period $u$). We need to show that so is $(\eta_0,\ldots,\eta_{m+1})$.
		
		It easily follows from the assumption that $u$ does not divide $m+1$, so we have some positive $k\in(u,2u)$ such that $k\equiv m+1\pmod u$. By Proposition~\ref{proposition:needed_qc}, it follows that $\eta_u\meet \eta_{m+1}=\eta_u\meet \eta_k<\eta_u\meet \eta_{2u}=\eta_0\meet \eta_u<\eta_0$. It follows that $\eta_{m+1}\not\geq \eta_0$. The not-inequality $\eta_{m+1}\not\leq\eta_0$ can be proven similarly, or follows immediately from the above argument by considering the partial orbit $(\eta_{m+1},\eta_m,\ldots,\eta_0)$ (and applying Proposition~\ref{proposition:qc_time_reversal}). In conclusion, $(\eta_0,\ldots,\eta_{m+1})$ is not a cycle, and not a spiral. We need to prove that it is not a comb.
		
		Suppose now that $m$ is odd, so $m+1=2c$. We need to show that $\eta_0\meet \eta_c=\eta_c\meet \eta_{2c}$. Let $k<u$ be such that $k\equiv c\pmod u$ (since $u$ does not divide $2c$, we have $k>0$, and $u$ divides none of $k,2k,c,2c$). Then clearly $c,k\neq 0$, $c\neq 2c$ and $k\neq 2k$, so (since $\eta_1,\ldots,\eta_{m+1})$ is a quasi-cycle with pseudo-period $u$), by Proposition~\ref{proposition:needed_qc}, $\eta_c\meet \eta_{2c}=\eta_k\meet \eta_{2c}=\eta_k\meet \eta_{2k}$. Again by Proposition~\ref{proposition:needed_qc}, $\eta_0\meet \eta_c=\eta_0\meet \eta_k$. Since $k<u$, $2k<2u\leq n$, we have $\eta_0\meet \eta_k=\eta_k\meet \eta_{2k}$, whence $\eta_0\meet \eta_c=\eta_c\meet \eta_{2c}$.
	\end{proof}
	
	\begin{corollary}
		\label{corollary:cycle}
		If $\eta=(\eta_0,\ldots,\eta_n)$ is a quasi-cycle, then it can be extended to a cycle.
	\end{corollary}
	\begin{proof}
		Let $u$ be the pseudo-period of $\eta$. The case of $n<2u$ is Proposition~\ref{proposition:baby_cycle}, so we may assume that $n\geq 2u$. Let $N$ be as in Corollary~\ref{corollary:qc_extension}, and let $(\eta_0,\ldots,\eta_N)$ be an arbitrary extension of $\eta$. We will show that $(\eta_0,\ldots,\eta_{N-1},\eta_0)$ is a partial orbit. Write $B$ for the subtree generated by $\{\eta_1,\ldots,\eta_{N-1}\}$. By Remark~\ref{remark:suff_p_aut}, it is enough to show that $\qftp(\eta_0/B)=\qftp(\eta_N/B)$. By Corollary~\ref{corollary:qc_extension}, we may assume without loss of generality that $N=n+1$. It follows easily from Proposition~\ref{proposition:needed_qc} that for every $i\in(0,N)$, $i\neq u$ we have $\eta_0\neq \eta_0\meet \eta_i=\eta_u\meet \eta_i=\eta_N\meet \eta_i\neq \eta_N$. Likewise, $\eta_0\neq \eta_0\meet \eta_u=\eta_u\meet \eta_{2u}=\eta_u\meet \eta_N\neq \eta_N$. Using this, the conclusion follows easily from Fact~\ref{fact:typesintrees}.
	\end{proof}
	
	The preceding corollary suggests the following question.
	\begin{question}
		\label{question:not_qc_ext}
		If $p$ is a finite partial automorphism of a tree, and $\eta$ is a quasi-cyclic orbit of $p$, does $p$ admit an extension $\bar p$ (to a finite partial automorphism, possibly of a larger tree) such that the orbit $\bar\eta$ of $\bar p$ containing $\eta$ is no longer quasi-cyclic?
	\end{question}
	
	Corollary~\ref{corollary:cycle} gives a positive answer under the additional assumption that $p$ has no other orbits besides $\eta$. In general, this seems plausible, but we do not know the answer.

	\begin{remark}
		\label{remark:can_do_anything}
		We believe that Corollary~\ref{corollary:cycle} can be extended to say the following (with a similar proof). If $\eta=(\eta_0,\ldots,\eta_n)$ is a quasi-cycle with pseudo-period $u$, then $\eta$ can be extended to exactly the following kinds of finite partial orbits:
		\begin{itemize}[nosep]
			\item
			a $k$-cycle,
			\item
			a quasi-cycle with pseudo-period $u$ (of arbitrary length),
			\item
			a quasi-cycle with pseudo-period $k$ (of arbitrary length),
			\item
			both an ascending and a descending spiral of spiral length $k$ (of arbitrary length),
			\item
			both an ascending and a descending spiral comb of spiral length $k$ (of arbitrary length),
			\item
			if $n<2u$, both an ascending and a descending spiral comb of spiral length $u$ (of arbitrary length).
		\end{itemize}
		(Here $k>n$ is a multiple of $u$.)\xqed{\lozenge}
	\end{remark}
	
	\section{Finding determined automorphisms}
	\label{section:Finding determined automorphisms}
	Recall that $\bT$ is the universal countable meet-tree.
	
	\begin{definition}
		\label{definition:initial_point}
		Given a finite partial automorphism $p$, we call a point $a\in\dom(p)$ an \emph{initial point of $p$} if $a\notin \range(p)$ or $a$ is in a cyclic orbit of $p$.\xqed{\lozenge}
	\end{definition}
	
	\begin{definition}
		Given sets $P,Q$, a partial function $p\colon P\to Q$, an element $a\in P$ and an integer $k$, we write $p^k(a)\isdefined$ if $a\in \dom(p^k)$, i.e.\ if $p^k(a)$ is well-defined. Otherwise, we write $p^k(a)\isundefined$ (if $p^k(a)$ is undefined).
		
		When $p$ is fixed in the context, and we have an orbit $\eta=(\eta_0,\ldots,\eta_n)$ of $p$, then we write also write $\eta_m\isdefined$ for $p^m(\eta_0)\isdefined$ and $\eta_m\isundefined$ for $p^m(\eta_0)\isundefined$.\xqed{\lozenge}
	\end{definition}
	
	We want to find a cofinal class of determined finite partial automorphisms. In order to do that, we introduce the following notion of a pseudo existentially closed partial automorphism of $\bT$. The idea is that ``any behavior'' of orbits of $p$ in extensions is already ``witnessed'' in $p$, and if possible, we want this to be witnessed at least twice (for technical reasons). This idea is made a bit more precise in Proposition~\ref{proposition:consequences}.
	\begin{definition}
		\label{definition:new_strong}
		Let $p$ be a finite partial automorphism of $\bT$. We say that $p$ is \emph{pseudo existentially closed} (\emph{PEC}) (in $\bT$) if it satisfies the following condition.
		
		For every extension $q\supseteq p$ to a partial automorphism of $\bT$, every triple $\eta_0,\mu_0,\zeta_0$ of initial points of $p$ (cf.\ Definition~\ref{definition:initial_point}) in orbits $\eta,\mu,\zeta$ (respectively) of $p$, and every $m_1,m_2>0$ such that $q^{m_1}(\mu_0)\isdefined,q^{m_2}(\zeta_0)\isdefined$, there exist (strictly) positive integers $m_1',m_2',m_1'',m_2''$ such that:
		\begin{itemize}
			\item
			if $\mu_{m_1}\isundefined$\footnote{note that here, $\mu_{m_1}=p^{m_1}(\mu_0)$, and analogously in the following bullets}, then $m_1'\neq m_1''$,
			\item
			the triples $(\eta_0,q^{m_1}(\mu_0),q^{m_2}(\zeta_0))$, $(\eta_0,\mu_{m_1'},\zeta_{m_2'})$ and $(\eta_0,\mu_{m_1''},\zeta_{m_2''})$ are well-defined, and all have the same quantifier-free type (in $\bT$; in particular, they have the same order type),
			\item
			if there is some $k$ such that $\mu_0\meet q^k(\mu_0)\neq q^k(\mu_0)\meet q^{2k}(\mu_0)$ (in particular, $q^{2k}(\mu_0)\isdefined$), then for minimal such $k$, we have $m_1\equiv m_1'\equiv m_1''\pmod k$. (Notice that in Proposition~\ref{proposition:consequences}\eqref{proposition:consequences:not_qc} we show that in this case, $\mu$ must be a spiral and $k$ is its spiral length.)\xqed{\lozenge}
		\end{itemize}
	\end{definition}

	Now we show that all finite partial automorphisms admit ``PEC closure'' (note that it is not canonical, in particular it is not a closure operator, hence the scare quotes).
	\begin{proposition}
		\label{proposition:PEC_extensions}
		Every finite partial automorphism of $\bT$ can be extended to a finite pseudo existentially closed (PEC) partial automorphism.
	\end{proposition}
	\begin{proof}
		Take any partial automorphism $p$. We extend it in three steps:
		\begin{enumerate}
			\item
			Given an orbit $\eta$ of $p$, if in some extension $f\supseteq p$, $\eta$ is extended to an orbit which is not a quasi-cycle, then we extend $\eta$ so that it is witnessed already in the corresponding extension of $p$; we repeat that for each orbit of $p$.
			\item 
			For each triple $\eta_0,\mu_0,\zeta_0$ of initial points of $p$, each residue mod the spiral length $k$ of $\mu$ (if it exists, i.e.\ $\mu$ is a spiral or a comb) and each quantifier-free type $r(x,y,z)$, we check whether for some $f\supseteq p$ there are $m_1,m_2$ (with $m_1$ having the appropriate residue mod $k$, if $\mu$ is a spiral or a comb) such that $\models r(\eta_0,f^{m_1}(\mu_0),f^{m_2}(\zeta_0))$. If the answer is yes, we extend $\mu$ and $\zeta$ to already contain witnesses $\mu_{m_1'},\zeta_{m_2'}$ for that.
			\item 
			We repeat the previous step, only this time, checking whether it is possible to have two witnesses $(\mu_{m_1'},\zeta_{m_2'})$ and $(\mu_{m_1''},\zeta_{m_2''})$ with $m_1'\neq m_1''$.
		\end{enumerate}
		This procedure terminates: we never add any new orbits in any of the steps, so step 1 finishes. After step 1, the set of initial points is fixed (because there are no new orbits, and step 1 ensures that there can be no new cycles), as are the spiral lengths of all the relevant orbits, so in each of steps 2 and 3 there is only a finite number of conditions to check.
		
		It is fairly easy to check that after the three steps, we obtain a PEC partial automorphism. Step 1 guarantees that the minimal $k$ checked in the last bullet of Definition~\ref{definition:new_strong} is simply the spiral length of $\mu$. Step 2 provides the witness $(m_1',m_2')$ required in the second bullet. Finally, at this point, the hypothesis of the first bullet of Definition~\ref{definition:new_strong} implies that $m_1\neq m_1'$, so by the third step, we also have $(m_1'',m_2'')$ with $m_1'\neq m_1''$ (otherwise, we can just take $m_1''=m_1'$ and $m_2''=m_2'$).
%
%
	\end{proof}
	
	\begin{proposition}
		\label{proposition:consequences}
		If $p$ is PEC and $\eta,\mu$ are orbits of $p$, while $f\supseteq p$ is contained in some $\bar f\in \Aut(\bT)$ (in particular, if $f$ is a finite partial automorphism of $\bT$), then:
		\begin{enumerate}[label=(\arabic{*})]
			\item
			\label{proposition:consequences:not_qc}
			if the $f$-orbit containing $\eta$ is neither a quasi-cycle nor a cycle, then there is some $k$ such that $\eta_{2k}\isdefined$ and $\eta_0\meet \eta_k\neq \eta_k\meet \eta_{2k}$ (note that even if $\eta$ is not a comb, by Proposition~\ref{proposition:spirals}(3), it follows that the minimal such $k$ is the spiral length of $\eta$),
			\item
			\label{proposition:consequences:orbit_type}
			if the $f$-orbit containing $\eta$ is a $k$-cycle, $k$-spiral (ascending or descending), $k$-comb (ascending or descending) or a quasi-cycle (respectively), then so is $\eta$ (respectively),
			\item
			\label{proposition:consequences:distinct_orbits}
			if $\eta,\mu$ are distinct, then so are the $f$-orbits extending them.
		\end{enumerate}
	\end{proposition}
	\begin{proof}
		We may assume without loss of generality that $f=\bar f\in \Aut(\bT)$. Write $\bar\eta$ for the $f$-orbit extending $\eta$, enumerated as $(\bar\eta_i)_{i\in \mathbf Z}$ so that $\eta_0=\bar\eta_0$ and $f(\bar\eta_i)=\bar\eta_{i+1}$. Likewise, let $\bar\mu\supseteq \mu$ be the $f$-orbit containing $\mu$.
		
		For \ref{proposition:consequences:not_qc}, since $\bar\eta$ is not a quasi-cycle, there is some $k$ such that $\eta_0\meet \bar\eta_k\neq \bar\eta_k\meet \bar\eta_{2k}$. We may assume without loss of generality that this $k$ is minimal. We claim that $\eta_{2k}\isdefined$, whence $\bar \eta_k=\eta_k$ and $\bar\eta_{2k}=\eta_{2k}$ and we are done. Otherwise, suppose towards contradiction that $\eta_{2k}\isundefined$, and apply PEC with $m_1=2k$, $m_2=k$ (and $\eta=\mu=\zeta = \eta$) and conclude that there positive $m_1',m_2',m_1'',m_2''$ with $m_1'\equiv m_1''\equiv 2k\pmod k$ such that $\eta_{m_1'}\isdefined,\eta_{m_1''}\isdefined$. But since $m_1',m_1''$ are distinct and positive, it follows that $\max(m_1',m_1'')\geq 2k$, so $\eta_{2k}\isdefined$, a contradiction.
		
		For \ref{proposition:consequences:orbit_type}, suppose first that $\bar\eta$ is a $k$-spiral or a $k$-cycle, so $\eta_0$ is comparable to $\bar\eta_k$. We claim that $\eta_k\isdefined$, whence $\bar\eta_k=\eta_k$ and we are done. Indeed, by PEC (with $m_1=m_2=k$), there is some $m_1'>0$ such that $\eta_{m_1'}\isdefined$ and $\otp(\eta_0,\eta_{m_1'})=\otp(\eta_0,\bar\eta_k)$. But then $\eta_{m_1'}=\bar\eta_{m_1'}$ is comparable to $\eta_0=\bar \eta_0$, so by definition of a $k$-cycle or $k$-spiral, $m_1'\geq k$, so $\eta_{k}\isdefined$.
		
		Otherwise, if $\bar\eta$ is a $k$-comb, then it is not a quasi-cycle, so $\eta$ is not a quasi-cycle by \ref{proposition:consequences:not_qc}. Since $\bar\eta$ is not a cycle, nor a spiral, neither is $\eta$, so it is an $l$-comb, whence $\bar\eta$ is an $l$-comb, so $l=k$.
		
		Finally, if $\bar\eta$ is a quasi-cycle, then by Remark~\ref{remark:trivial_preservation}, so is $\eta$, completing the proof of \ref{proposition:consequences:orbit_type}.
		
		
		Finally, \ref{proposition:consequences:distinct_orbits} is immediate by PEC: if $\bar\eta_{m_1}=\bar\mu_{m_2}$, then we have some $m_1',m_2'$ such that $\eta_{m_1'}=\mu_{m_2'}$.
	\end{proof}

	\begin{remark}
		\label{remark:immediate_extension_stays_pec}
		It is easy to see that an immediate extension (in the sense of Definition~\ref{definition:immediate}) of a PEC automorphism of $\bT$ is itself PEC (because the initial points are the same, since there can be no new cycles, no new orbits, and we only extend existing orbits positively).\xqed{\lozenge}
	\end{remark}
	
	\begin{remark}
		\label{remark:no_qcycles}
		Note that if the answer to Question~\ref{question:not_qc_ext} is positive, then it follows by Proposition~\ref{proposition:consequences}\ref{proposition:consequences:orbit_type} that a PEC partial automorphism has no quasi-cyclic orbits. This would make the proof of Theorem~\ref{theorem:main} below a bit simpler --- namely, we could drop the final three paragraphs of the proof, and we would no longer need Proposition~\ref{proposition:needed_qc}.\xqed{\lozenge}
	\end{remark}
	
	The following Theorem is one of the main building blocks of the proof that the universal countable meet-tree has a generic automorphism.
	\begin{theorem}
		\label{theorem:main}
		Let $p$ be a PEC finite partial automorphism of $\bT$, and $A$ the meet-tree generated by the union of all orbits of $p$.
		
		Let $\xi=(\xi_0,\xi_1,\ldots,\xi_n)$ be an orbit of $p$ such that no non-cyclic orbit of $p$ has smaller length.
		
		Then if $p\cup \{(\xi_n, v)\}$ and $p\cup \{(\xi_n,w)\}$ are both extensions of $p$ to a partial automorphism of $\bT$, then $\tp(v/A)=\tp(w/A)$.
	\end{theorem}
	Note that in the context of the theorem, both $p\cup \{(\xi_n, v)\}$ and $p\cup \{(\xi_n,w)\}$ are immediate extensions of $p$ in the sense of Definition~\ref{definition:immediate}, which will allow us to apply Lemma~\ref{lemma:sufficient_for_determined}.
	\begin{proof}
		Put $B:=\langle \range(p)\rangle$, $p_v:=p\cup\{(\xi_n,v)\}$ and $p_w:=p\cup\{(\xi_n,w)\}$. Note that the conclusion is trivially true if $\xi$ is a cycle, so in the rest of the proof, we assume that $\xi$ is not a cycle, and so $\xi_{n+1}\isundefined$. Furthermore, $\tp(v/\range(p))=p(\tp(\xi_n/\dom(p)))=\tp(w/\range(p))$, and hence $\tp(v/B)=\tp(w/B)$.

		\begin{claim}\label{claim:typesnotequal}
			For every $\eta_i\in \dom(p)$, we have $\tp(\eta_i/B) \neq \tp(v/B)$.
		\end{claim}
		\begin{clmproof}
			Suppose towards contradiction that $\tp(\eta_i/B) = \tp(v/B)$. First, we show that in this case $p\cup \{(\xi_n,\eta_i)\}$ is a partial automorphism. Let $\bar c$ enumerate $\dom(p)$. Let $\varphi(x,\bar c)$ be a quantifier-free formula. We need to show that if $\varphi(\xi_n,\bar c)$ holds, then so does $\varphi(\eta_i,p(\bar c))$. But since $p_v$ is a partial automorphism, $\varphi(\xi_n,\bar c)$ holds if and only if $\varphi(v,p(\bar c))$ holds. But this is true if and only if $\varphi(x,p(\bar c))\in \tp(v/B)=\tp(\eta_i/B)$, i.e.\ $\varphi(\eta_i,p(\bar c))$ holds.
			
			Now, if $\eta=\xi$ and $p\cup \{(\xi_n,\eta_i)\}$ is a partial automorphism, it means that $\xi$ can be turned into a cycle in an extension of $p$. Since we have assumed $\xi$ is not already a cycle, this contradicts Proposition~\ref{proposition:consequences}\ref{proposition:consequences:orbit_type}. Otherwise, if $\eta\neq \xi$ and $p\cup \{(\xi_n,\eta_i)\}$ is a partial automorphism, we contradict Proposition~\ref{proposition:consequences}\ref{proposition:consequences:distinct_orbits}.
		\end{clmproof}
		
		Let $v':= \max_{b\in B} b\meet v$ and likewise, $w':=\max_{b\in B} b\meet w$.

		\begin{claim}
			\label{claim:sufficient_condition}
			$\tp(v/A)=\tp(w/A)$ if and only if $v'$ and $w'$ have the same order type over $A_{\leq b}$ for some $b\geq v',w'$.
		\end{claim}
		\begin{clmproof}
			We have that $v'=\max_{a\in A} a\meet v$. Indeed, otherwise, we would have $a\meet v>v'$ for some $a\in A\setminus B$, and hence also for some $\eta_0\in \dom(p)\setminus \range(p)$ (namely, any such that $\eta_0\geq a$), we would have $\eta_0\meet v>v'$, and so, by Corollary~\ref{corollary:meet_above_prime}, $\tp(\eta_0/B)=\tp(v/B)$, contradicting Claim~\ref{claim:typesnotequal}. Analogously, $w'=\max_{a\in A} a\meet w$.
			
			Furthermore, $v=v'$ holds if and only if $w=w'$: these hold if and only if for some $b\in B$, the type $\tp(v/B)=\tp(w/B)$ implies $b\geq x$. Thus, the conclusion follows by Fact~\ref{fact:typesintrees}.
		\end{clmproof}
		
		If $v'\in B$, then $\tp(v/B)$ contains a formula $\varphi(x,B)$ which says ``$x\geq v'$ and $v'=\max_{b\in B} x\meet b$''. Since $\tp(v/B)=\tp(w/B)$, we conclude $w$ also satisfies this formula, so $v'=w'$, and in particular, they have the same order type over $A$, so by Claim~\ref{claim:sufficient_condition}, $\tp(v/A)=\tp(w/A)$, completing the proof.
		
		By the preceding paragraph, for the rest of the proof, we may assume that $v'\notin B$. Note that as we have seen in the proof of Claim~\ref{claim:sufficient_condition}, we have that $v=v'$ if and only $w=w'$, and analogously, since $v'\notin B$, we have $w'\notin B$. We will treat separately the cases when $v=v'\notin B$ and $v\neq v'\notin B$, but first, we make some observations that apply to both of them.
		
		Put $\alpha:=\max\{b\in B\mid b\leq v' \}$ ($-\infty$ if the set is empty), and choose $\beta\in B$ to be minimal above $v'$, so $\beta\geq v'$ and $(v',\beta)\cap B=\emptyset$. Then $v'\in (\alpha,\beta)$ and $(\alpha,\beta)\cap B=\emptyset$. Note that $\beta\meet v=v'$. Indeed, $\beta\meet v\leq v'$ by definition of $v'$, and since $\beta \geq v'$ and $v\geq v'$, also $\beta\meet v\geq v'$. Since $\tp(v/B)=\tp(w/B)$, it follows that $w'\in (\alpha,\beta)$ and $w'=w\meet \beta$. Notice that analogously, for every $b\in B$ such that $b\geq \beta$, we have also $b\wedge v=v'$ and $b\wedge w=w'$.
		
		In the analysis of the two cases, we will use the $\alpha$ and $\beta$ heavily.
		
		\begin{claim}
			\label{claim:form_of_elts}
			Every element of $A\cap (\alpha,\beta)$ is of the form $\eta_0\meet \beta$ for some $\eta_0\in \dom(p)\setminus B$.
		\end{claim}
		\begin{clmproof}
			Let $a\in A\cap (\alpha,\beta)$. Then for some $\eta,\mu,i,j$ we have $a=\eta_i\meet\mu_j$. Since $a<\beta$, we have $\eta_i\meet \beta\geq a$ and $\mu_j\meet \beta \geq a$. By Fact~\ref{fact:lemmaxyz}, only one of these inequalities can be strict. Suppose without loss of generality that $\eta_i\meet \beta = a$. Since $a\in (\alpha,\beta)$, it follows that $a\notin B$, so (because $\beta\in B$) $\eta_i\notin B$, so also $\eta_i\notin \range(p)$, whence $\eta_i\in \dom(p)$, which completes the proof.
		\end{clmproof}
		
		Since $v',w'\in (\alpha,\beta)$, by Claim~\ref{claim:form_of_elts} and Claim~\ref{claim:sufficient_condition}, it is enough to show that for every $\eta_0\in \dom(p)\setminus B$, at least two of the following hold:
		\begin{enumerate}[label=(\roman*)]
			\item
			$\eta_0\meet \beta\leq v'$ if and only if $\eta_0\meet \beta\leq w'$,
			\item
			$\eta_0\meet \beta\geq v'$ if and only if $\eta_0\meet \beta\geq w'$,
			\item
			$\eta_0\meet \beta = v'$ if and only if $\eta_0\meet \beta=w'$,
		\end{enumerate}
		(because any two of these conditions imply that $v',w'$ have the same order type over $(\alpha,\beta)\cap A$, and hence also over $A_{\leq \beta}$.)
		
		Now, if $\eta_0\wedge \beta\notin (\alpha,\beta)$, then it is clear that these equivalences hold (because $v',w'\in (\alpha,\beta)$). Hence, for the rest of the proof, let us fix an arbitrary $\eta_0\in \dom(p)\setminus B$ such that $\eta_0\wedge \beta\in (\alpha,\beta)$. The goal is to show that for this $\eta_0$, two of the above three equivalences hold.
		
		\textbf{Case 1:} $v' = v \notin B$.
		
		If $\eta_0\meet \beta\geq v'$, then trivially $\eta_0 \geq v$. Since $p_v$ extends $p$ (as a partial automorphism of $\bT$), $v=p_v^{n+1}(\xi_0)$ and $\xi_{n+1}\isundefined$, by Definition~\ref{definition:new_strong} (applied to $m_1=m_2=n+1$, $q=p_v$), there are $i<j\leq n$ such that $\xi_i,\xi_j\leq \eta_0$, and hence they are comparable. Since $\xi$ is not a cycle, it follows that it must be a spiral (see Definition~\ref{definition:orbit_types} to recall the orbit types). Thus, if $\xi$ is not a spiral, then for all $\eta_0\in \dom(p)\setminus B$ we have $\eta_0\meet \beta<v'$ and analogously, $\eta_0\meet \beta<w'$.
		
		By the preceding paragraph, we may assume that $\xi$ is spiral, of spiral length $k$ for some $k$. We will show (*) that $\eta_0\meet \beta\geq v$ (equivalently, $\eta_0\geq v$) if and only if $\eta_0\geq \xi_{n+1-k}$ (and so, by symmetry, $\eta_0\meet\beta\geq w$ if and only if $\eta_0\geq \xi_{n+1-k}$).
		
		Suppose first that $\xi$ is descending. Then $v< \xi_{n+1-k}$, so if $\eta_0\geq \xi_{n+1-k}$, then trivially $\eta_0\geq v$. Conversely, if $\eta_0\geq v$, then by Definition~\ref{definition:new_strong} (and the equality $v=p_v^{n+1}(\xi_0)$), for some $i\equiv n+1\pmod k$, we have $\eta_0\geq \xi_i$. Since $i\equiv n+1\equiv n+1-k\pmod k$ and $\xi$ is descending, $\xi_i\geq\xi_{n+1-k}$ and so $\xi_{n+1-k}\leq\eta_0$.
		
		Otherwise, suppose $\xi$ is ascending. Then $v>\xi_{n+1-k}$, so if $\eta_0\geq v$, then trivially $\eta_0\geq \xi_{n+1-k}$. Conversely, if $ \eta_0\not\geq v$, then again by Definition~\ref{definition:new_strong}, there is some $i\equiv n+1\pmod k$, $i\leq n$, such that $\xi_i\not\leq \eta_0$, and as before, we have $\xi_i\leq \xi_{n+1-k}$, so $\xi_{n+1-k}\not\leq \eta_0$, completing  the proof of (*), which shows (ii).
		
		To complete the proof in this case, it is enough to show (iii), i.e.\ that given $\eta_0\in \dom(p)\setminus B$, we have $\eta_0\meet \beta=v$ if and only if $\eta_0\meet \beta=w$. We will show that these equalities are impossible (given that $\xi$ is spiral and $v=v'\notin B$, and hence $w=w'\notin B$). Notice that for some $\mu_i\in \range(p)$ (so without loss of generality, $i>0$) we have that $\mu_i\geq \beta$. Since $\beta=\mu_i\meet\beta>\beta\meet\eta_0$, by Fact~\ref{fact:lemmaxyz}, it follows that $\beta\meet\eta_0=\mu_i\meet\eta_0$.
		
		Suppose $\xi$ is a descending $k$-spiral and $\eta_0\meet \beta=v$. By Definition~\ref{definition:new_strong} (with $q=p_v$, $m_1=n+1$, $m_2=i$), there is some $j\equiv n+1\pmod k$ and $i'>0$ such that $\eta_0\meet \mu_{i'}=\xi_j$. Since $\xi$ is descending, $\xi_j>v$, so $\eta_0\meet\mu_{i'}>\eta_0\meet \mu_i=v$, and by Fact~\ref{fact:lemmaxyz}, we have that $v=\eta_0\meet \mu_i=\mu_i\meet \mu_{i'}$. Since $i,i'>0$, it follows that $v\in B$, a contradiction.
		
		Now, suppose $\xi$ is an ascending $k$-spiral and $\eta_0\meet \beta=v$. Arguing as in the preceding paragraph, we conclude that for some $j\equiv n+1\pmod k$ and a positive $i'$ we have $\eta_0\meet \mu_{i'}=\xi_j$. We may assume without loss of generality that $i'$ is maximal (i.e.\ for all $i''>i'$ there is no $j'\equiv n+1\pmod k$ such that $\eta_0\meet \mu_{i''}=\xi_{j'}$).
		
		We claim that $i'\leq i-k$. Indeed, put $l:=i'-i$ and let $f\in \Aut(\bT)$ be an arbitrary extension of $p_v$; note that the $f$-orbit of $\xi_j$ is still an ascending $k$-spiral, and also that $f^k(\xi_j)\leq f^{n+1-j}(\xi_j)=v\leq \eta_0$, so in particular, $f^k(\xi_j)\not\leq \mu_{i'}$ --- otherwise, $f^k(\xi_j)\leq \mu_{i'}\meet \eta_0=\xi_j$, which would contradict the assumption that the $f$-orbit of $\xi_j$ is an ascending $k$-spiral. Since $\eta_0\meet \mu_i=v>\xi_j=\eta_0\meet \mu_{i'}$, we have $\mu_i\meet \mu_{i'}=\xi_j$, so $\mu_{i'}\meet f^l(\mu_{i'})=f^l(\xi_j)$. Now, $\mu_{i'}\geq \xi_j,f^l(\xi_j)$, so the latter two are comparable, whence $k$ divides $l$ (by Proposition~\ref{proposition:spirals}); this implies that $f^l(\xi_j)$ and $f^k(\xi_j)$ are comparable. On the other hand, since $f^k(\xi_j)\not\leq \mu_{i'}$ and $f^l(\xi_j)\leq \mu_{i'}$, we have $f^l(\xi_j)<f^k(\xi_j)$, so $l<k$. Since $k$ divides $l$ and $l$ cannot be $0$, it follows that $l\leq -k$.
		
		It follows that in fact $j=n+1-k$. Otherwise, if $j< n+1-k$, then $ \xi_{n+1-k}>\xi_j=\mu_i\meet\mu_{i'}$. On the other hand, $\mu_{i-k}\geq \xi_{n+1-k}$, so $\mu_{i-k}\meet \mu_i\geq \xi_{n+1-k}$, and hence $\mu_{i-k}\meet \mu_{i'}=\xi_j$. It would follow that $\mu_i\meet \mu_{i'+k}=\xi_{j+k}$. Now, $\mu_i\meet \eta_0=v>\xi_{j+k}$, so $\mu_{i'+k}\meet\eta_0=\xi_{j+k}$, contradicting the maximality of $i'$.
		
		Thus, we have that $i'\leq i-k$ and $\mu_{i'}\meet \eta_0=\xi_{n+1-k}$. It follows that $\mu_{i'+k}\meet \eta_k=v$ ($\mu_{i'+k}$ is well-defined because $i'\leq i-k$ and $\mu_i$ is well-defined, while $\eta_k$ is well-defined because $\xi$ is the shortest non-cyclic orbit). But $\mu_{i'+k}\meet \eta_k\in B$, contradicting $v\notin B$, so we conclude that $\eta_0\meet \beta \neq v,w$, showing (iii) in Case~1.
		
		Thus, we have completed the proof in Case~1 (i.e.\ under the assumption that $v=v'\notin B$). The following is the last remaining case.
		
		\textbf{Case 2:} $v' < v$ and $v ' \notin B$. Recall that we have fixed some $\eta_0\in \dom(p)\setminus B$ such that $\eta_0\meet \beta\in (\alpha,\beta)$, and we need to show that it compares to $v'$ in the same way as it compares to $w'$.

		\begin{claim}
			\label{claim:elts_if_notequal}
			In the case we are considering, $\eta_0\in (\alpha,\beta)$.
		\end{claim}
		\begin{clmproof}
			We have that $\eta_0\meet \beta\in (\alpha,\beta)$. It trivially follows that $\eta_0\meet \beta$ and $v'$ have the same order type over $B_{\leq \beta}$. Let $\bar b\in B$ be such that $\eta_0\meet \bar b=\max_{b\in B} \eta_0\meet b$. Then $\eta_0\meet \beta \leq \eta_0\meet \bar b$. If the inequality is strict, then by Fact~\ref{fact:lemmaxyz}, we have $\eta_0\meet \beta= \bar b\meet \beta\in B\cap (\alpha,\beta)$, which contradicts the fact that $B\cap (\alpha,\beta)=\emptyset$. Otherwise, $\eta_0\meet \beta=\eta_0\meet \bar b\in (\alpha,\beta)$, so $\max_{b\in B} \eta_0\meet b$ has the same order type over $B_{\leq \beta}$ as $v'$. Since, by Claim~\ref{claim:typesnotequal}, $\tp(v/B)\neq \tp(\eta_0/B)$, by Fact~\ref{fact:typesintrees} and the inequality $v\neq v'$, we conclude that $\eta_0= \max_{b\in B} \eta_0\meet b=\eta_0\meet \beta$, and therefore $\eta_0=\eta_0\meet \beta\in (\alpha,\beta)$.
		\end{clmproof}

		\begin{claim}
			\label{claim:no_new_meets}
			For all orbits $\mu$ of $p$ and $k>0$, we have that $\eta_0=v \meet \mu_k$ if and only if there are $l,r>0$ such that $\xi_n \meet \mu_{k-1} = \xi_{l-1} \meet \mu_{r-1}$ and $\eta_0 = \xi_l \meet \mu_r$.
			
			In particular, $\eta_0=v'$ if and only if $\eta_0=w'$.
		\end{claim}
		\begin{clmproof}
			Suppose $\eta_0=v\meet \mu_k$. By Definition~\ref{definition:new_strong} (applied to $q=p_v$, $m_1=n+1$, $m_2=k$), we have some $l,r$ such that $\eta_0=\xi_l\meet \mu_r$, and they clearly satisfy the right-hand side. The converse is immediate (just apply $p_v$ to both sides of the first equality).
			
			For the ``in particular'', just note that $v'=v\meet \mu_k$ for some $\mu,k$ (namely, any such that $\mu_k\in B$ and $\mu_k\geq \beta$) and $w'=w\meet \mu_k$ for the same $\mu,k$ (because $\tp(v/B)=\tp(w/B)$) and apply the first part.
		\end{clmproof}
		
		By Claim~\ref{claim:elts_if_notequal} and Claim~\ref{claim:no_new_meets}, it is enough to show that we have $\eta_0\leq v'$ if and only if $\eta_0\leq w'$. Note that since $v'\in (\alpha,\beta)$, this implies that $v'$ and $\eta_0$ are comparable. Note also that since $v'<v$ and $\eta_0<\beta$, it is easy to see that $\eta_0\leq v'$ if and only if $\eta_0<v$.

		Suppose $\xi$ is not a quasi-cycle, then by Proposition~\ref{proposition:consequences}\ref{proposition:consequences:not_qc}, for some (minimal) $k$, we have that $\xi_0\meet \xi_k\neq \xi_k\meet \xi_{2k}$ (so by semilinearity, there is a strict inequality). We claim that under this assumption, $\eta_0\leq v'$ (since $\eta_0<\beta$, equivalently, $\eta_0< v$) if and only if $\eta_0\leq \xi_{n+1-k}\meet \xi_{n+1-2k}$.
		
		If $\xi_0\meet \xi_k<\xi_{k}\meet \xi_{2k}$, then by applying $p_v^{n+1-2k}$ to both sides, $\xi_{n+1-k}\meet \xi_{n+1-2k}<\xi_{n+1-k}\meet v$, which makes one implication trivial. In the other direction, if $v> \eta_0$, then since $\xi_{n+1}\isundefined$, by Definition~\ref{definition:new_strong} (applied to $q=p_v$, $m_1=m_2=n+1$), there are distinct $i_1,i_2\equiv n+1\pmod k$ such that $\xi_{i_1},\xi_{i_2}\geq \eta_0$, whence $\xi_{i_1}\meet \xi_{i_2}\geq \eta_0$. We may assume without loss of generality that $i_1<i_2$, whence $i_1\leq n+1-2k$ (because $i_1\equiv i_2\equiv n+1\pmod k$ and $i_2\leq n$). By Proposition~\ref{proposition:needed_comb}\eqref{it:needed_comb_1b}, it follows that $\xi_{i_1}\meet \xi_{i_2}\leq \xi_{n+1-2k}\meet \xi_{n+1-k}$, and hence $\eta_0\leq \xi_{n+1-2k}\meet \xi_{n+1-k}$.
		
		Otherwise, suppose $\xi_0\meet \xi_k>\xi_k\meet \xi_{2k}$. If $\eta_0< v$, then by Definition~\ref{definition:new_strong} (with $q=p_v$, $m_1=m_2=n+1$), there is some $i\equiv n+1\pmod k$ such that $\xi_i> \eta_0$, whence $\xi_i\meet v\geq \eta_0$. Then by Proposition~\ref{proposition:needed_comb}(2), $\xi_i\meet v=\xi_{n+1-k}\meet v<\xi_{n+1-2k}\meet \xi_{n+1-k}$. Conversely, suppose $\eta_0\not< v$. Then, again by Definition~\ref{definition:new_strong} (having in mind that $\xi_{n+1}\isundefined$), there are distinct $i_1,i_2\equiv n+1\pmod k$ such that $\eta_0\not\leq \xi_{i_1},\xi_{i_2}$, so in particular, $\eta_0 \not\leq \xi_{i_1}\meet \xi_{i_2}$; since $i_1,i_2\leq n$ and $i_1,i_2\equiv n+1\pmod k$, we have $\max(i_1,i_2)\leq n+1-k$, and thus by Proposition~\ref{proposition:needed_comb}\eqref{it:needed_comb_2b}, $\xi_{i_1}\meet \xi_{i_2}\geq \xi_{n+1-k}\meet \xi_{n+1-2k}$, whence $\xi_{n+1-k}\meet \xi_{n+1-2k}\not\geq \eta_0$.
		
		We are left with the case when $\xi$ is a quasi-cycle. Note that by Proposition~\ref{proposition:consequences}\ref{proposition:consequences:orbit_type}, it follows that $(\xi_0,\ldots,\xi_n,v)$ is also a quasi-cycle. Let $u$ be the pseudo-period of $\xi$ (cf.\ Definition~\ref{definition:pseudo-period}). We will show that
		$v>\eta_0$ if and only if there is some positive $k\neq n+1-u$ such that $\xi_k\meet \xi_{n+1-u}\geq \eta_0$.
		
		Suppose $v>\eta_0$. Considering $\xi_{n+1}\isundefined$, by Definition~\ref{definition:new_strong}, there are distinct positive $k, k'\leq n$ such that $\xi_k,\xi_{k'}>\eta_0$. Note that this implies that $v\meet \xi_k\geq \eta_0$, so also $v\meet \xi_{n+1-u}\geq \eta_0$ (by Proposition~\ref{proposition:qc_time_reversal}), and in particular, $\xi_{n+1-u}\geq \eta_0$. Since $k\neq k'$, we may assume without loss of generality that $k\neq n+1-u$, and then clearly $\xi_k\meet \xi_{n+1-u}\geq \eta_0$.
		
		Now, suppose $0<k\neq n+1-u$ and $\xi_k\meet \xi_{n+1-u}\geq \eta_0$. Then by Proposition~\ref{proposition:needed_qc}, $\xi_n\meet \xi_{k-1}=\xi_{n-u}\meet \xi_{k-1}$, so also $v\meet \xi_k=\xi_{n+1-u}\meet \xi_k\geq \eta_0$, and hence $v> \eta_0$.
		
		This finishes the proof in Case~2, completing the proof of Theorem~\ref{theorem:main}.
	\end{proof}

	\begin{corollary}
		\label{corollary:PEC_is_determined}
		A PEC partial automorphism of $\bT$ is determined (in the sense of Definition~\ref{definition:determined_KT}).
	\end{corollary}
	\begin{proof}
		It is enough to check that every PEC partial automorphism $p$ satisfies the hypothesis of Lemma~\ref{lemma:sufficient_for_determined}. Take $p=f_0\subsetneq f_1\subsetneq\ldots\subsetneq f_n\subsetneq f_{n+1}$ such that $f_i\subsetneq f_{i+1}$ is immediate for $i\leq n$ and a positively strict extension $g\supseteq f_n$. Let $a=\nu_m$ be the sole element of $\dom(f_{n+1})\setminus \dom(f_n)$. Put $v:=f_{n+1}(a), w:=g(a)$. Then by Remark~\ref{remark:immediate_extension_stays_pec}, $f_n$ is PEC, so by Theorem~\ref{theorem:main}, $\tp(v/A)=\tp(w/A)$ (where $A$ is the substructure of $\bT$ generated by the orbits of $f_n$), so, by ultrahomogeneity of $\bT$, there is some $\sigma\in \Aut(\bT/A)$ such that $\sigma(v)=w$. It follows that $\sigma \circ f_{n+1}\circ \sigma^{-1}\subseteq g$, and so $\tau_n:=\sigma$ witnesses that the hypothesis of Lemma~\ref{lemma:sufficient_for_determined} is satisfied, which completes the proof.
	\end{proof}

	\begin{corollary}
		\label{corollary:cap_jep_hp}
		The class $\cK^1$ of finite meet-trees with a single partial automorphism has CAP and JEP.
	\end{corollary}
	\begin{proof}
		Given any $(B',p')\in \cK^1$, we may assume that $B'\subseteq \bT$, and then extend $p'$ to a finite partial automorphism $p''$ of $\bT$ such that $B'\subseteq \dom(p'')$. By Proposition~\ref{proposition:PEC_extensions}, we can extend $p''$ to a PEC partial automorphism $p$ of $\bT$. If we take $B$ to  be the tree generated by $\dom(p)\cup \range(p)$, then by Corollary~\ref{corollary:PEC_is_determined} and Corollary~\ref{corollary:determined_is_AB_trees}, if $p$ is a PEC automorphism of $\bT$, then $(B,p)$ is an amalgamation base in $\cK^1$, and clearly $(B',p')\subseteq(B,p)$, which shows CAP.
		
		To see JEP, take any $(A,p_A),(B,p_B)\in \cK^1$. We may assume without loss of generality that $A\cap B=\emptyset$. Let $v$ be a new element. Then let $A\vee_v B:=A\cup B\cup \{v\}$, ordered in such a way $a\meet b=v$ for all $a\in A$ and $b\in B$. Clearly, $p_A\cup p_B$ is a partial automorphism of $A\vee_v B$ and $(A,p_A)$ and $(B,p_B)$ embed into $(A\vee_v B,p_A\cup p_B)$.
	\end{proof}
	
	The following theorem, along with Corollary~\ref{corollary:no_generic_pair}, completes the \hyperref[mainthm]{Main Theorem}.
	\begin{theorem}
		\label{theorem:gen_automorphism_exists}
		The universal meet-tree $\bT$ has a generic automorphism.
	\end{theorem}
	\begin{proof}
		This is immediate by Corollary~\ref{corollary:cap_jep_hp} and Fact~\ref{corollary:sufficient_for_generics}.
	\end{proof}
	
	\begin{remark}
		\label{remark:works_for_bounded_arity}
		It seems like it might be possible to use largely the same methods to show that for each $k>0$, $\bT_k$ has a generic automorphism. More specifically, we believe that the analogue of Theorem~\ref{theorem:main} is true with essentially the same proof (after we replace the notion of PEC by the appropriate variant for $\bT_k$), which will yield a cofinal class of \emph{determined} automorphisms of $\bT_k$. On the other hand, Proposition~\ref{proposition:no_AP_bdd} shows that the analogue of Lemma~\ref{lemma:AP_in_trees_with_automorphisms} fails if $k>1$, so we cannot simply use Lemma~\ref{lemma:determined_is_amalgamation_base} to conclude that the class of determined automorphisms witnesses CAP, as we did in the proof of Corollary~\ref{corollary:cap_jep_hp}.
		
		Instead, one could try to show that every finite partial automorphism of $\bT_k$ can be extended to one which is PEC (in $\bT_k$) and has the property that its unique strict extension satisfies the hypothesis of Remark~\ref{remark:sufficient_for_AP_bdd}, which (together with Lemma~\ref{lemma:determined_is_amalgamation_base} and a proof of Remark~\ref{remark:sufficient_for_AP_bdd}) would yield CAP. For $k=2$, this seems straightforward, but the general case appears to be more difficult --- it is plausible that for $k>2$ we might have EAP but not CAP, leading to a negative answer to Question~\ref{question:EAP_not_CAP}.
		
		Note that in any case we easily get JEP for the class of finite $k$-ary meet-trees for all $ k $: for $ k>1 $ it is exactly as in the proof of Corollary \ref{corollary:cap_jep_hp}, and for $ k=1 $ (i.e., linear orders) it is also quite easy. Thus, for all $ k $, the automorphism group of $ \bT_k $ has a dense conjugacy class. 
		\xqed{\lozenge}
	\end{remark}
	
	\begin{remark}
		\label{remark:works_for_linear}
		The caveats mentioned in Remark~\ref{remark:works_for_bounded_arity} do not seem to apply in the case of $k=1$ (see Remark~\ref{remark:AP_linear_aut}), so we can slightly adjust the proof of Theorem~\ref{theorem:gen_automorphism_exists} to recover the fact that $(\mathbf Q,<)$ has a generic automorphism.\xqed{\lozenge}
	\end{remark}
	
	\begin{remark}
		\label{remark:works_for_lex_trees}
		Consider the class of finite meet-trees expanded with a lexicographic ordering (i.e.\ a total order $\unlhd$, extending the tree order, such that if $b\unlhd a$ and $a\meet a'>a\meet b$, then $b\unlhd a'$). It is easy to see that it is a Fraïssé class, and one can ask whether the limit of this class has a generic automorphism.
		
		The proof of Lemma~\ref{lemma:AP_in_trees_with_automorphisms} seems to adapt to this context in a straightforward manner, so the main difficulty seems to lie with the analogue of Theorem~\ref{theorem:main}.
		
		Note that the orbit analysis for trees with a lexicographic ordering is much simpler: there are no nontrivial cycles (there can be fixed points), which implies that the spiral length can only be equal to $1$ (so the only spirals are just monotone sequences, the only spiral combs are $1$-combs), and the pseudo-period of a quasi-cycle is always one (so all quasi-cycles are ``fans'').\xqed{\lozenge}
	\end{remark}

	\printbibliography
\end{document}